\documentclass[a4paper,11pt,oneside]{article}
\usepackage{graphicx}
\usepackage{todonotes}
\usepackage{tikz-cd}

\usepackage{colonequals}

\usepackage[style=ext-numeric,
maxbibnames=99,  
maxcitenames=2,
giveninits=true, 
sortcites=true,  
backend=biber]{biblatex}
\usepackage[nottoc,notlot,notlof]{tocbibind}   

\usepackage{amsmath}
\usepackage{amsthm}
\usepackage{amssymb}
\usepackage{mathtools}
\usepackage{authblk}
\usepackage[shortlabels]{enumitem}
\usepackage{fullpage}
\usepackage{esint}

\usepackage{wrapfig}

\usepackage{lmodern}

\usepackage{tikz}

\usetikzlibrary{shapes.geometric, arrows, arrows.meta, calc, decorations.pathreplacing}
\usepackage{empheq}

\usepackage[most]{tcolorbox}
\definecolor{bananamania}{rgb}{0.98, 0.91, 0.71}
\definecolor{indigo}{rgb}{0.0, 0.25, 0.42}
\definecolor{pistachio}{rgb}{0.58, 0.77, 0.45}
\definecolor{lightgrey}{rgb}{0.9, 0.9, 0.9}
\definecolor{cornflowerblue}{rgb}{0.39, 0.58, 0.93}
\newtcolorbox{box1}{colback=bananamania, colframe=black, fontupper=\color{indigo}}
\newtcolorbox{box2}{colback=blue!20!white, colframe=black}
\newtcolorbox{box3}{colback=pistachio, colframe=black, fontupper=\color{black}}
\newtcolorbox{cbox}{colback=lightgrey, colframe=black}
\newtcolorbox{rbox}{colback=red!75!green!50, colframe=black}

\usepackage{xcolor}

\usepackage{hyperref}
\hypersetup{colorlinks, linkcolor=blue, urlcolor=red, filecolor=green, citecolor=blue}

\usepackage{nicematrix}

\theoremstyle{plain}
\newtheorem{theorem}{Theorem}[section]
\newtheorem{corollary}[theorem]{Corollary}
\newtheorem{lemma}[theorem]{Lemma}
\newtheorem{proposition}[theorem]{Proposition}
\newtheorem{definition}[theorem]{Definition}
\newtheorem{assumption}[theorem]{Assumption}
\newtheorem{problem}{Problem}

\theoremstyle{remark}
\makeatletter
\def\@endtheorem{\hfill$\diamond$\endtrivlist\@endpefalse }
\makeatother
\newtheorem{remark}[theorem]{Remark}

\newcommand{\N}{\mathbb{N}}
\newcommand{\R}{\mathbb{R}}

\newcommand{\Z}{\mathbb{Z}}

\newcommand{\dd}{\mathop{}\!\mathrm{d}}

\newcommand\e{\varepsilon}
\newcommand\dist{\operatorname{dist}}
\newcommand\sym{\operatorname{sym}}
\newcommand{\SO}[1]{\operatorname{SO}(#1)}

\newcommand\per{{\operatorname{per}}}

\newcommand\II{{\operatorname{I\!I}}}
\newcommand\deform{v} 

\newcommand{\wto}{\rightharpoonup}

\DeclarePairedDelimiter{\abs}{\lvert}{\rvert}
\DeclarePairedDelimiter{\norm}{\lVert}{\rVert}

\def\eff{{\rm eff}}
\def\iso{{\rm iso}}
\def\cA{\mathcal{A}}
\def\cT{\mathcal{T}}
\def\cP{\mathcal{P}}
\def\cQ{\mathcal{Q}}
\def\cG{\mathcal{G}}
\def\cE{\mathcal{E}}
\def\cI{\mathcal{I}}

\def\cN{\mathcal{N}}

\def\cV{\mathcal{V}}
\def\cW{\mathcal{W}}

\def\pRed{\mathcal{P}_3^{\text{red}}}

\def\plateh{\mathsf{\eta}}

\newcommand{\Qhom}{Q_{\textup{hom}}^{\gamma} }
\newcommand{\bQhom}{\mathbf{Q}_{\textup{hom}}^{\gamma} }
\newcommand{\Qhomh}{Q_{\textup{hom}}^{\gamma,h} }

\newcommand{\Beff}{B_{\textup{eff}}^{\gamma} }
\newcommand{\Beffh}{B_{\textup{eff}}^{\gamma,h} }

\newcommand{\Beffc}{\widehat{B}_{\textup{eff}}}

\newcommand{\Qhomc}{\widehat{Q}}

\newcommand{\Bc}{\widehat{B}}
\newcommand{\dgrad}{\nabla_H} 
\newcommand{\dhess}{\nabla\nabla_H} 
\newcommand{\corr}{\nabla_\gamma \vartheta}  
\newcommand{\energy}{\cE^\gamma} 
\newcommand{\tenergy}{\widetilde{\cE}^\gamma} 
\newcommand{\dtenergy}{\widetilde{\cE}_H^{\gamma,h}} 

\newcommand{\lagint}{\cI_H^{1}} 

\newcommand{\macroquad}[1]{\cG_\dom^H\left[#1\right]} 

\newcommand{\microgrid}{\cT_h^\Box} 
\newcommand{\macrogrid}{\cT_H} 

\newcommand{\dom}{S} 

\newcommand{\mac}{s} 

\newcommand{\Rsym}[1]{\R^{#1\times#1}_{\operatorname{sym}} }

\title{Finite element discretizations of bending plates with prestrained microstructure}

\author{Klaus Böhnlein\footnote{\href{mailto:klaus.boehnlein@tu-dresden.de}{klaus.boehnlein@tu-dresden.de}},
Stefan Neukamm\footnote{\href{mailto:stefan.neukamm@tu-dresden.de}{stefan.neukamm@tu-dresden}},
and Oliver Sander\footnote{\href{mailto:oliver.sander@tu-dresden.de}{oliver.sander@tu-dresden.de}}}

\affil{Faculty of Mathematics, Technische Universit\"at Dresden}

\begin{document}

\maketitle

\begin{abstract}
  We investigate a finite element discretization of an elastic bending-plate model with an effective prestrain.
  The model has been obtained via homogenization and dimension reduction by \textcite{boehnlein2023homogenized}.
  Its energy functional is the $\Gamma$-limit of a three-dimensional nonlinear microstructured
  elasticity functional. In the derived effective model, the microstructure is incorporated
  as a local corrector problem,
  a system of linear elliptic partial differential equations posed on a three-dimensional
  representative volume element. 
  The discretization uses Discrete Kirchhoff Triangle elements for the macroscopic bending-plate problem
  on a mesh of scale $H$, and first-order Lagrange elements for the microscopic corrector problem on an axis-aligned mesh of scale $h$.
  We show that the discretized model $\Gamma$-converges to the continuous one as $(h,H)\to 0$,
  provided that there exists a microstructure mesh such that the
  elasticity tensor is Lipschitz continuous on each mesh element.
  This extends earlier results by \textcite{rumpf2023two} to prestrained composites.
  Our argument does not require any rate of convergence for the microscopic discretization error.
  As a corollary, we also obtain convergence when $h \to 0$ and $H \to 0$
  consecutively, and we prove that these limit processes commute.
\end{abstract}
\smallskip

\textbf{Keywords:} nonlinear elasticity, microstructure, bending plates, isometry constraint, prestrain, DKT elements, $\Gamma$-convergence

\smallskip

\textbf{MSC-2020:} 74B20 35B27 74Q05 74S05

\tableofcontents

\section{Introduction}

In this paper we investigate a finite element discretization of the nonlinear
bending model for prestrained, elastic plates derived in \cite{boehnlein2023homogenized}.
The model is given by the effective non-convex energy functional
\begin{equation}\label{eq:Energy_intro}
  \cE^\gamma:H^2_{\iso}(S;\R^3)\to\R,
  \qquad
  \cE^\gamma(\deform) \colonequals \int_S \Qhom\big(\mac,\II_\deform(\mac)-\Beff(\mac)\big)\dd\mac,
\end{equation}
where $S\subset\R^2$ denotes a Lipschitz domain representing the reference configuration of the plate, and
$\II_{\deform}$ is the second fundamental form associated with the deformation $\deform :S \to \R^3$.
The admissible space $H^2_{\iso}$ of bending deformations consists of all maps with square-integrable
second weak derivative that satisfy the local isometry constraint
\begin{equation}
  (\nabla\deform)^\top\nabla\deform=I_{2\times 2}
  \label{IsometryConstraint}
\end{equation}
almost everywhere. Note that this space is nonlinear, which is a central difficulty of this type of models.

In~\cite{boehnlein2023homogenized}, the energy~\eqref{eq:Energy_intro} is derived
as a $\Gamma$-limit from a three-dimensional, nonlinear elasticity model with a prestrained
locally periodic microstructure, described by a multiplicative decomposition of the deformation gradient into an elastic
and an inelastic part. In the limit of zero thickness
and zero periodicity length, the material properties are captured
by two objects: A space-dependent quadratic positive-definite form $\Qhom(\mac,\cdot):\Rsym2\to[0,\infty)$, $s \in S$,
called the effective bending stiffness, and a matrix field $\Beff : S \to \Rsym2$ called the
effective prestrain. These quantities can be computed from the specification
of the three-dimensional model by solving particular microstructure problems, which are essentially
linear elasticity problems for the microstructure on representative volume elements.

For various reasons it is of interest to understand the relationship
between the microstructure and the energy minimizers of the effective bending model.
Generally, the presence of an effective prestrain $\Beff$ leads to curved energy minimizers,
whose curvature and bending directions depend on the structure of the composite in intricate ways.
In \cite{boehnlein2023homogenized} this has been studied for plates with globally periodic
microstructure, and without boundary conditions. In this case, $\Qhom$ and $\Beff$ do not depend
on $s\in S$, and energy minimizers are cylindrical deformations. These deformations
can be computed exactly by minimizing
a smooth function on a two-dimensional set.

In the present paper we extend the scenario to composites with a space-dependent microstructure,
and we also take Dirichlet boundary conditions for the plate into account. In this case,
the reduction to a finite-dimensional minimization problem is impossible, and numerical
approximation is required. However, this is challenging for several reasons:

\begin{enumerate}[(a)]
  \item \label{item:challenge_H2}
  The minimization of \eqref{eq:Energy_intro} constitutes an $H^2$-conforming problem
  that features the nonlinear isometry constraint $(\nabla\deform)^\top\nabla\deform=I_{2\times 2}$.
  
  \item \label{item:challenge_bilevel}
  The problem has a bilevel structure: In addition to the discretization of \eqref{eq:Energy_intro}
  on a macroscopic scale $H$, one needs to discretize the microstructure problems
  on a second scale $h$ to obtain approximations for $\Qhom(s,\cdot)$ and $\Beff(s)$.
  
  \item \label{item:challenge_regularity}
  Typical composite materials exhibit discontinuous elastic moduli on the microscopic scale.
  Also, the effective quantities $\Qhom(s,\cdot)$ and $\Beff(s)$ can be discontinuous in $s$.
  Consequently, the convergence analysis can assume only low regularity.
\end{enumerate}

Regarding~\ref{item:challenge_H2}, various finite element techniques have been proposed to approximate
deformations in $H^2_\iso(S;\R^3)$, but none of them fully captures the geometric nonlinear constraint.
A truly conforming discretization would require an $H^2$-conforming finite element space
that satisfies the pointwise isometry constraint~\eqref{IsometryConstraint} everywhere. Unfortunately, such a space of discrete isometric deformations is currently unknown.
Requiring the isometry~\eqref{IsometryConstraint} in standard finite element spaces stiffens the discrete model,
and one therefore has to introduce nonconformities to achieve a controlled softening.
One way is to enforce the isometry~\eqref{IsometryConstraint} weakly via a penalty approach.
This has been proposed for a mixed finite element
discretization of a Reissner--Mindlin plate model based on MINI and Crouzeix--Raviart elements in~\cite{bartels2013finite}, and for a spline-based model in~\cite{mohan2021minimal}.
Discontinuous Galerking (DG) methods are an alternative way to soften the discrete model,
and have been used in~\cite{bonito2022ldg,bonito2023numerical,bonito2024gamma,bonito2020discontinuous,bartels2024error,bonito2021dg}.

In this paper, we apply a discretization introduced by \citeauthor{bartels2013approximation}~\cite{bartels2013approximation,bartels2022simulating}. It is based on the
Discrete Kirchhoff Triangle (DKT) finite element~\cite{braessfem,batoz1980study},
and it enforces the isometry constraint exactly, but only at the nodes of the mesh.
It was used originally for bending plates featuring a flat reference configuration~\cite{bartels2013approximation},
but was extended by \textcite{rumpf2022finite} to the case of thin elastic shells
described by parametrized surfaces.
Extensions to bilayer plate models with a given fixed spontaneous curvature tensor
were studied in \cite{bonito2024gamma,bartels2017bilayer,palus2024finite}. Several works
\cite{bartels2020finite,bartels2022stable,mohan2021minimal} focused on
homogeneous and isotropic material layers in particular, which are a special case of the
non-homogeneous and non-isotropic materials of this paper.

Various authors have noted that the second fundamental form $\II_\deform$ that appears
in bending-plate models can be replaced by the deformation Hessian $\nabla^2 \deform$
if the deformations are constrained to be isometries
\cite{bartels2022stable,bartels2017bilayer,rumpf2023two,mohan2021minimal,bonito2024gamma}.
This simplifies the model, because a nonlinear quantity is replaced by a linear one.
We show that such a reformulation is possible even in the presence of prestrain,
and that it leads to replacing the effective prestrain $\Beff$ by
the product $n_{\deform}\otimes\Beff$ with the surface normal $n_\deform$.
We observe similar theoretical and numerical advantages as for the reformulated problem
without prestrain.

With regard to~\ref{item:challenge_bilevel}, the bilevel nature of the problem,
convergence of the model without prestrain has been shown by \textcite{rumpf2023two} under additional regularity assumptions on the microstructure.
Similar to our approach, their fully discrete model combines a DKT-discretization of
the macroscopic problem with a first-order Lagrange finite element discretization
of the microscopic problem, and they show $\Gamma$-convergence of the resulting fully discrete model
to the continuous energy as the microscopic- and macroscopic discretization parameters $h$ and $H$ tend to zero.
Their analysis is based on abstract error analysis for heterogeneous multiscale methods
in~\cite{ming2005analysis}, and crucially requires a linear rate of convergence
in the microscopic discretization parameter $h$. The latter needs $H^2$-regularity
of the microscopic linear elasticity problem and thus in \cite{rumpf2023two} Lipschitz-regularity
of the elastic moduli on the microscale is required. This rules out
applying their theory to many typical composites, which frequently involve jumping coefficients.

To deal with irregular composites as described in~\ref{item:challenge_regularity},
the proof of $\Gamma$-convergence of the fully discrete model we present in this paper
only requires qualitative convergence of the
microscopic problem as $h\to 0$, and it does not depend on any specific features of the
microstructure discretization. This makes the proof flexible in the sense that, besides
the first-order finite element discretization on axis-aligned grids used in this paper,
it also allows other discretizations such as spectral or particle methods for the microstructure
problems. Global convergence always follows as long as the microstructure discretization
converges (with or without a rate).

Our generalization allows us to treat composites that are discontinuous and locally Lip\-schitz
on the level of the elements of the mesh on scale $h$. To obtain the required qualitative
convergence of the scheme as $h\to 0$, we use $\Gamma$-convergence methods
which do not need $H^2$-regularity.
We also make no assumption of homogeneity in the out-of-plane direction of the material
in the original three-dimensional model. In particular, microstructured multilayer plates are covered.
This is important for applications such as simulating self-shaping wooden bilayers \cite{boehnlein2024}, as plates whose material and prestrain are homogeneous
in the out-of-plane direction only lead to flat equilibrium shapes.

The paper is organized as follows. In Section~\ref{sec:model} we recall the underlying
three-dimensional model and its limit under simultaneous homogenization and dimension reduction.
We explain the effective bending stiffness and prestrain, and the corrector problems
used to define them. We also present the previously mentioned reformulation
of the prestrain energy functional in terms of the Hessian.
In Section~\ref{sec:MicroNumerics} we introduce a discretization of the microscopic corrector problem
and prove its $\Gamma$-convergence.
In Section~\ref{sec:MacroDiscretization} we first recall DKT-elements and introduce the fully discrete model. Our main result is Theorem~\ref{T:main}, which proves
$\Gamma$-convergence of this fully discrete model based only on very weak assumptions
regarding the microstructure discretizations. As corollaries we obtain that
the discretizations also converge under the individual limits $h \to 0$ and $H \to 0$,
and that these consecutive limit processes commute.

\section{The effective bending model}
\label{sec:model}

We briefly describe the initial three-dimensional microstructured and prestrained nonlinear
elasticity model, and the derived effective two-dimensional bending plate.
See \cite{boehnlein2023homogenized} for more details.

A key property of the three-dimensional nonlinear model will be that the linearized material behavior is given,
for every macroscopic point $s \in S$ and microstructure point $y$
in the representative volume element $\Box$, by an elasticity tensor
$\mathbb L(s,y)\in\operatorname{Lin}(\R^{3\times 3};\R^{3\times 3})$,
which only acts on the symmetric part of its argument and is itself symmetric in the sense that
\begin{equation}
  \label{eq:L_symmetry}
  \big\langle \mathbb L(s,y) G, F \big\rangle
  =
  \big\langle \mathbb L(s,y) F, G \big\rangle
  =
  \big\langle \mathbb L(s,y)\sym F, \sym G \big\rangle
  \qquad
  \forall F,G\in\R^{3\times 3}.
\end{equation}
Here $\langle F,G \rangle \colonequals \sum_{i,j=1}^3F_{ij}G_{ij}$ denotes the standard
Frobenius inner product on matrices, with corresponding matrix norm $\abs{\cdot}$.
As we explain below, the effective coefficients of the derived two-dimensional model only depend on $\mathbb L$.
The crucial feature is that we assume that $\mathbb L$ is uniformly elliptic in the sense that
there exists a constant $c_0 > 0$ with
\begin{equation}
  \forall F,G\in\R^{3\times 3}\,:\qquad
  \big\langle \mathbb L(s,y)F, G \big\rangle
  \leq
  c_0\abs{\sym F} \abs{\sym G},
  \qquad
  \big \langle \mathbb L(s,y)G, G \big \rangle
  \geq
  \frac1{c_0}\abs{\sym G}^2 \label{eq:Lbounds}
\end{equation}
for all $s\in S$ and $y\in\Box$.

\subsection{The three-dimensional model}
\label{sec:3d_model}

Consider a three-dimensional nonlinear elasticity model on a cylindrical domain
of thickness $\plateh$ that describes a composite material that is locally in-plane periodic
on a length scale $\e$. More precisely,
let $\Omega_\plateh\colonequals S\times(-\frac \plateh2,\frac \plateh2)$,
where $S\subseteq\R^2$ is an open, bounded, connected Lipschitz domain,
and consider the elastic energy functional
\begin{equation}\label{def:ene}
  \mathcal E^{\e,\plateh}_\text{3d}(\deform)\colonequals
  \int_{\Omega_\plateh} W\Big(\mac,\{\tfrac{\mac}{\e}\},\tfrac{x_3}{\plateh},\nabla \deform\big(I_{3\times 3}-\plateh B(\mac,\{\tfrac{s}{\e}\},\tfrac{x_3}{\plateh})\big)\Big)\dd(\mac,x_3),
\end{equation}
where $s \colonequals (s_1,s_2)\in S$ and $x_3\in (-\frac \plateh2,\frac \plateh2)$
denote the in-plane and out-of-plane coordinates, respectively.
The map $\deform:\Omega_\plateh\to\R^3$ denotes a three-dimensional deformation
of $\Omega_\plateh$.
For $y' \in \R^2$, we define $\{y'\} \colonequals ( \{y'_1\}, \{y'_2\} ) \in [0,1)^2$ as the vector of fractional parts of the components of $y'$, i.e., $\{y'_i\} \colonequals y'_i - \lfloor y'_i \rfloor$ for $i=1,2$.
This induces a periodic continuation of $W$.

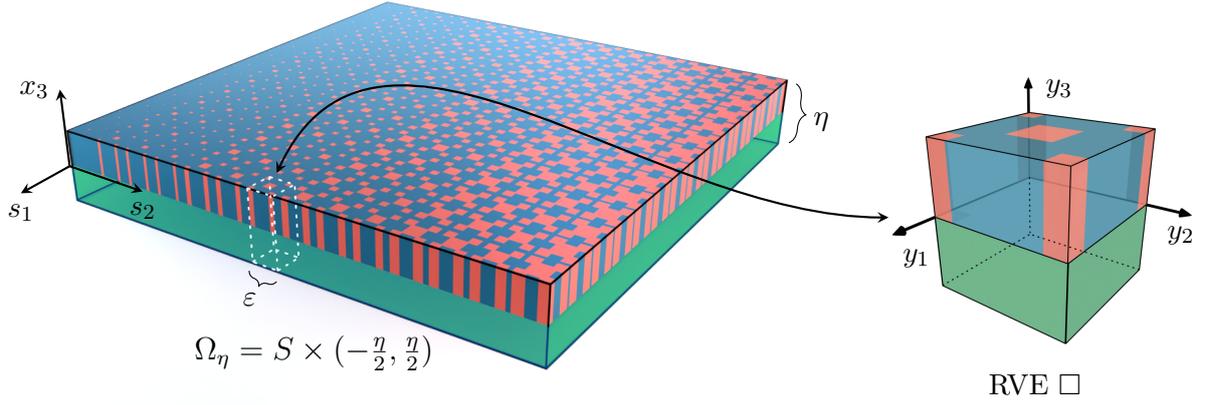
\begin{figure}
  \begin{center}
    \begin{tikzpicture}[scale=1.2,>=stealth]
      \node[anchor=south west] at (0.0,0) {\includegraphics[width=0.6\textwidth]{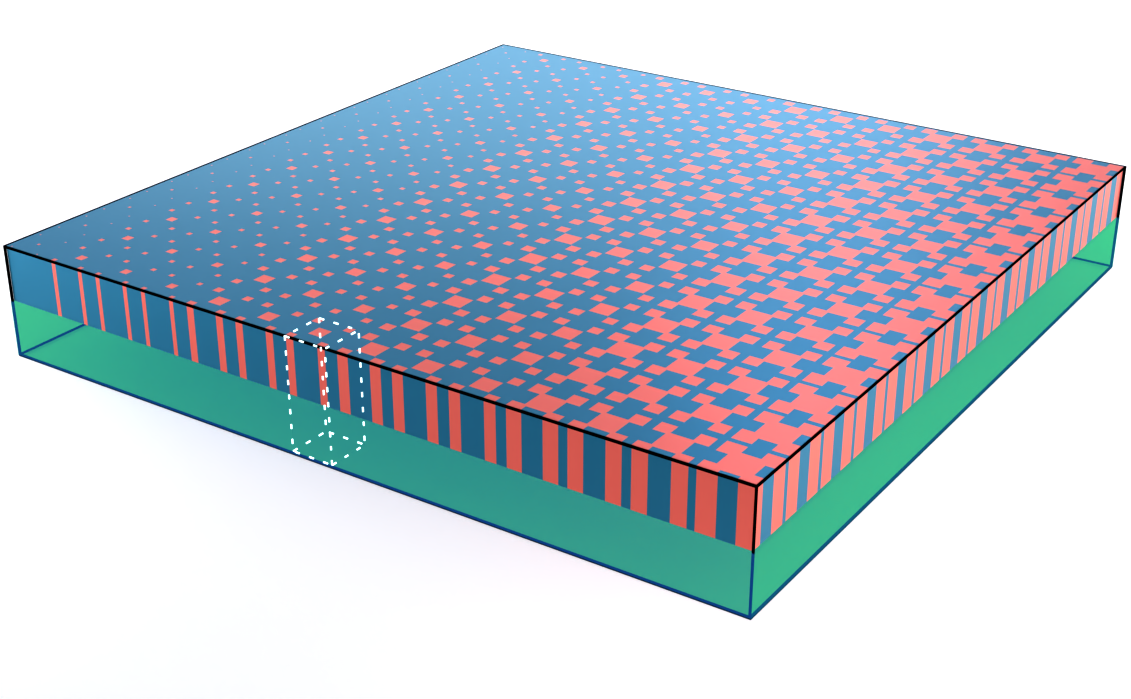}};;
      
      \coordinate (P) at (2.0,2.35,-1);
      
      \begin{scope}[
        shift={(0,2.25,0)}, 
        x={(1cm,-0.35cm)},        
        y={(-0.15cm,1cm)},        
        z={({cos(220)},{sin(220)})}, 
        scale=0.8 
        ]
        \coordinate (O) at (0.35,0.95,0); 
        \draw[->, thick, black] (O) -- (1.35,0.95,0) node[below] {$s_2$};
        \draw[->, thick, black] (O) -- (0.35, 0.95, 1.0) node[below] {$s_1$};
        \draw[->, thick, black] (O) -- (0.35, 1.995, -0.035) node[left] {$x_3$};
      \end{scope}
      
      \draw [decorate,decoration={brace, amplitude=5pt,raise=8pt}] (7.9,3.85) -- (7.82,3.2)
      node [right, midway, xshift=12pt, yshift=-4pt] {$\eta$};
      \draw [decorate,decoration={brace, amplitude=4pt,raise=6pt}] (2.5,1.8) -- (2.2,1.92)
      node [below, xshift=-2pt, yshift=-10pt] {$\e$};
      \node at (4,2,3) {\large$\Omega_\eta=S\times(-\tfrac\eta2,\tfrac\eta2)$};
      
      \draw[<->, thick, black] (P)+(0.05,0.1,0.0) to[out=65, in=180] (10.3,3.5,3);
      
      \node[anchor=south west, rotate=0] at (12.5,4,9) {\includegraphics[width=0.25\textwidth]{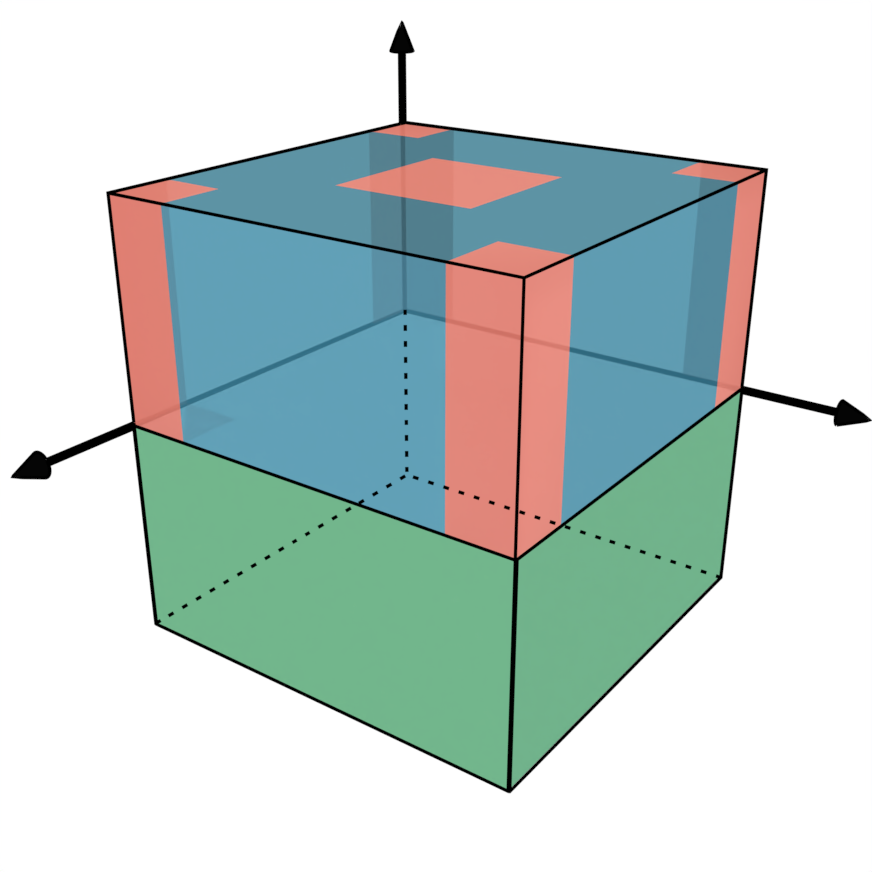}};;
      \node at (9.45,1.9,0) {$y_1$};
      \node at (12.35,2.15,0) {$y_2$};
      \node at (11.4,4.15,1) {$y_3$};
      
      \node at (10.75,0.5,0) {RVE $\Box$};
    \end{tikzpicture}
  \end{center}
  \caption{The two-dimensional domain $S$ and the representative volume element (RVE) $\Box$ are used to model the microstructure of the material on three-dimensional reference domain $\Omega_\eta$.}
  \label{fig:RVE}
\end{figure}

The two functions $W$ and $B$ that appear in~\eqref{def:ene} are a stored energy
density and a prestrain matrix field, respectively. Since they describe a material with a
microstructure, both depend on two types of coordinates: The global
coordinates $s \in S \subset \R^2$, and additional microstructure coordinates
$(\{\frac{s}{\epsilon}\}, \frac{x_3}{\plateh}) \in \Box\colonequals(0,1)^2\times(-\frac12,\frac12)$
(Figure~\ref{fig:RVE}). The microstructure domain $\Box$ can be regarded as a three-dimensional
representative volume element that describes the heterogeneity of the material
on the composite scale.  Formally, we get
\begin{align*}
  W : S\times\Box\times\R^{3\times 3}\to[0,\infty],
  \qquad \text{and} \qquad
  B : S\times\Box\to\R^{3\times 3}.
\end{align*}

The derivation in~\cite{boehnlein2023homogenized} of an effective bending model
with an elasticity tensor $\mathbb L$ that satisfies~\eqref{eq:Lbounds}
rests on the following assumptions:
\begin{assumption}[Three-dimensional energy density]
  The function $W$ is Borel-measurable, and there exist a constant $c_0>0$ and
  a monotone function $r:[0,\infty)\to[0,\infty]$  satisfying $\lim_{\delta\to0}r(\delta)=0$
  such that for all $s\in S$ and for almost every $y \in \Box$ the following properties hold:
  \begin{enumerate}[label=({W}\arabic*)]
    \item $W$ is frame indifferent:
    $$W(s,y,RF)=W(s,y,F)\qquad\text{for all $F\in\R^{3\times 3}$, $R\in \SO 3$.}$$
    
    \item \label{ass:nondegeneracy}
    $W$ is non-degenerate, i.e.:
    \begin{alignat*}{2}
      W(s,y,I_{3\times 3})&=0,&&\\
      W(s,y,F) &\geq \frac1{c_0} \dist^2(F,\SO 3) & \qquad & \text{for all $F\in\R^{3\times 3}$,}\\
      W(s,y,F) &\leq c_0 \dist^2(F,\SO 3)  &        & \text{for all $F\in\R^{3\times 3}$ with $\dist^2(F,\SO 3)\leq \frac1{c_0}$,}
    \end{alignat*}
    where $\dist^2(F,\SO 3)\colonequals\min_{R\in\SO 3}\abs{F-R}^2$.
    
    \item There is a quadratic form $Q(s,y,\cdot):\R^{3\times 3}\to\R$ such that
    \begin{equation*}
      \abs[\big]{W(x_3,y,I_{3\times 3}+G)-Q(x_3,y,G)}
      \leq
      \abs{G}^2r(\abs{G})\qquad\text{for all $G\in\R^{3\times 3}$.}
    \end{equation*}
  \end{enumerate}
\end{assumption}

Regarding the prestrain, we assume (with the same constant for simplicity):

\begin{assumption}[Three-dimensional prestrain]
  The matrix field $B$ is Borel-measurable, and there exists a constant $c_0>0$
  such that for all $s\in S$, the prestrain $B$ is integrable, and
  \begin{equation*}
    \int_\Box|B(s,y)|^2\dd y\leq c_0.
  \end{equation*}
\end{assumption}

Given these assumptions, for every $s \in S$ and $y \in \Box$ the elasticity tensor
$\mathbb L(s,y)\in\operatorname{Lin}(\R^{3\times 3};\R^{3\times 3})$
can be derived from the quadratic form $Q(s,y,\cdot)$ via the identity
\begin{equation*}
  \forall F,G\in\R^{3\times 3}\,:\quad
  \big\langle \mathbb L(s,y) F, G \rangle
  \colonequals
  \frac12\big(Q(s,y,F+G)-Q(s,y,F)-Q(s,y,G)\big).
\end{equation*}
Note that such an $\mathbb{L}$ indeed only acts on the symmetric part of its argument,
as required in~\eqref{eq:L_symmetry}.
It is proved in \cite[Lemma~2.7]{neukamm2012rigorous} that
the non-degeneracy of the energy density $W$ (Assumption~\ref{ass:nondegeneracy}) implies
the continuity and ellipticity of $\mathbb L$ as stated in~\eqref{eq:Lbounds}.

\subsection{Passing to the limit}

We replace~\eqref{def:ene} by an effective model for the limit $(\e,\plateh) \to (0,0)$.
In~\cite{boehnlein2023homogenized}, we showed $\Gamma$-convergence of the scaled
energy functional $\plateh^{-3}\mathcal E^{\e,\plateh}_\text{3d}$ to the two-dimensional prestrained bending energy
\begin{equation}\label{eq:Energy}
  \cE^\gamma:H^2_{\iso}(S;\R^3)\to\R,
  \qquad
  \cE^\gamma(\deform) \colonequals \int_S \Qhom\big(\mac,\II_\deform(\mac)-\Beff(\mac)\big)\dd\mac
\end{equation}
on the space
\begin{equation*}
  H^2_{\iso}(S;\R^3)
  \colonequals
  \Big\{\deform\in H^2(S;\R^3)\,:\, (\nabla\deform)^\top\nabla\deform=I_{2\times 2}\,\Big\}.
\end{equation*}
This combines homogenization $\e\to 0$ and dimension reduction $\plateh\to0$, and requires that the ratio $\plateh/\e$ converges to a finite limit $\gamma\in (0,\infty)$.
The effective energy depends on this limit.

\begin{remark}
  We note that the result of \cite{boehnlein2023homogenized} covers models
  that are slightly more general than the situation specified in Section~\ref{sec:3d_model}.
  In particular, we assume in our paper that the composite is locally $\Z^2$-periodic everywhere,
  whereas in \cite{boehnlein2023homogenized} the periodicity cell could change with every $s\in S$.
  Unlike the model of \cite{rumpf2023two}, we can handle prestrain
  and materials that depend on the out-of-plane variable $y_3$.
\end{remark}

The $\Gamma$-convergence result of \cite{boehnlein2023homogenized} also covers
affine displacement boundary conditions: Let $\Gamma_D\subset\partial S$ be relatively open and non-empty, and let $\widetilde{u}_D\in H^2_{\iso}(S;\R^3)$ be a bending deformation
with unit normal $n_{\widetilde{u}}\colonequals\partial_1\widetilde u\times \partial_2\widetilde u$.
Then one may consider clamped boundary conditions for the three-dimensional model
by restricting $\mathcal E^{\e,\plateh}_\text{3d}$ to the set
\begin{equation*}
  \cA^{\plateh}
  \colonequals
  \Big\{\deform \in H^1(\Omega_{\plateh};\R^3)\,:\,\deform(s,x_3)=(1-\plateh)\widetilde{u}_D(s)+x_3 n_{\widetilde{u}_D}(s)\text{ a.e.~on }\Gamma_D\times\big(\textstyle{-\frac\plateh2,\frac\plateh2}\big)\Big\}.
\end{equation*}
In \cite{boehnlein2023homogenized} it is shown under appropriate assumptions on $\Gamma_D$
that $\mathcal E^{\e,\plateh}_\text{3d}\vert_{\cA^\plateh}$ converges to
the effective energy $\mathcal E^\gamma$ restricted to the set
\begin{equation}\label{eq:BC}
  \cA
  \colonequals
  \Big\{\deform\in H^2_{\iso}(S;\R^3)\,:\,\deform=\widetilde{u}_D\text{ and }\nabla\deform=\nabla\widetilde{u}_D\text{ a.e.~on $\Gamma_D$}\Big\}.
\end{equation}
Showing this convergence requires that
\begin{equation}\label{prop:density}
  \text{$\cA\cap C^3(\overline S;\R^3)$ is $H^2$-dense in $\cA$.}
\end{equation}
A sufficient condition for this denseness is for instance:
\begin{assumption}
  \label{ass:domain}
  The midplane $S$ is a convex and bounded Lipschitz domain with piecewise $C^1$-boundary,
  the Dirichlet boundary $\Gamma_D\subset\partial S$ is a relatively open, non-empty line segment,
  and $\nabla\widetilde{u}_D$ is constant on~$\Gamma_D$.
\end{assumption}
We note that although the density of $C^\infty(\overline S;\R^3)\cap H^2_{\rm iso}(S;\R^3)$ in $H^2_{\iso}(S;\R^3)$
is known \cite{pakzad2004sobolev, HornungIsometries},
showing that the Assumptions~\ref{ass:domain} imply \eqref{prop:density} is nontrivial.
It is proved in \cite{bartels2023nonlinear}, and extended in \cite{griehl2024bending}
to slightly more general boundary conditions.

The effective problem is now a constrained minimization problem:
\begin{problem}[Effective problem]
  \label{prob:macminorig}
  \begin{equation*}
    \text{Minimize $\energy(\deform)$ subject to $\deform\in\cA$}.
  \end{equation*}
\end{problem}
With the direct method it is easy to see that this minimization problem admits a
(typically non-unique) solution if $\cA$ is a non-empty subset of $H^2_{\iso}(S;\R^3)$ that is closed with respect to weak convergence in $H^2$.
In particular, this is the case when we define $\cA$ by prescribing boundary conditions as in \eqref{eq:BC}.

\subsection{The effective properties of the two-dimensional plate}
\label{sec:homogenization_formulas}
In this section we define the homogenized quadratic form $Q_{\hom}^\gamma$
and the effective prestrain $B_{\eff}^\gamma$ that appear in the $\Gamma$-limit $\mathcal E^\gamma$.
Both quantities depend on the parameter $\gamma \colonequals \lim \frac{\plateh}{\e} \in(0,\infty)$.
The definitions of $Q_{\hom}^\gamma$ and $B^\gamma_{\eff}(\mac)$ are local in $\mac$:
For a fixed macroscopic point $\mac\in S$, the quantity $Q_{\hom}^\gamma(\mac,\cdot)$
only depends on the map $\Box\times\R^{3\times 3}\ni(y,G)\mapsto Q(\mac,y,G)$
of the three-dimensional model.
Likewise, $B^\gamma_{\eff}(\mac)$ depends only on the prestrain
map $\Box\ni y\mapsto B(\mac,y)$.

The precise definitions involve solutions of an elliptic boundary value problem
on the domain~$\Box$.
To formulate this problem we introduce the space of in-plane-periodic $H^1$-functions%
\footnote{Except for the constraint $\int_\Box \varphi \dd y = 0$,
this is the space denoted by $H^1_\gamma$ in \cite{boehnlein2023homogenized}.}
\begin{multline*}
  H^1_{\per}(\Box;\R^3)
  \colonequals
  \Big\{\varphi\in H^1(\Box;\R^3)
  \,:\, \int_\Box \varphi \dd y = 0,
  \\
  \text{$\varphi(0,\cdot,y_3)=\varphi(1,\cdot,y_3)$ and $\varphi(\cdot,0,y_3)=\varphi(\cdot,1,y_3)$}
  \\
  \text{in the sense of traces, for a.e.\ $y_3\in(-\tfrac12,\tfrac12)$}\,\Big\}.
\end{multline*}
The solutions of the elliptic problem are sums of in-plane affine and periodic microdeformations.
Define the embedding
\begin{equation*}
  \iota : \R^{2\times 2} \to \R^{3\times 3},
  \qquad
  \iota(G)\colonequals\sum_{i,j=1}^2 G_{ij}e_i\otimes e_j,
\end{equation*}
where $e_1, e_2, e_3$ is the canonical basis of $\R^3$. With this, define the function space
\begin{equation*}
  \mathcal{V}
  \colonequals
  \Big\{ \psi \in H^1(\Box;\R^3) : \ \psi(y) = \iota(M)y + \varphi(y),
  \ M \in \Rsym2,\  \psi \in H^1_\per(\Box;\R^3) \Big \}.
\end{equation*}
The space $\mathcal V$ contains displacements of the representative volume element $\Box$
that correct macroscopic strains of the form $y\mapsto\iota(y_3 G)$ (with $G\in\R^{2\times 2}_{\sym}$),
thereby ensuring force equilibrium on the level of the RVE. Such displacements
are called correctors. We denote them by $\vartheta(\mac, G)$ to indicate
their dependence on the macroscopic strain $G$ and the macroscopic position $\mac$.
For any given $\mac\in S$ and $G\in\Rsym2$ the corrector $\vartheta(s,G)$ is defined as the unique solution of the corrector problem
\begin{equation}\label{eq:corrector_equation}
  \int_{\Box} \Big\langle\mathbb L(\mac,y) \left(\iota(y_3 G)+\corr(\mac,G)\right), \nabla_\gamma v \Big\rangle\dd y=0 \quad \forall v\in \cV
\end{equation}
in the space $\cV$,
where $\nabla_\gamma\varphi \colonequals (\partial_{y_1}\varphi,\partial_{y_2}\varphi,\tfrac1\gamma\partial_{y_3}\varphi)$ is a transversally scaled gradient.

The following lemma from \cite{boehnlein2023homogenized} establishes the existence and uniqueness
of the corrector.

\begin{lemma}[Existence and uniqueness of a corrector field {\cite[Lemma~2.23]{boehnlein2023homogenized}}]
  For every $G\in \R^{2\times 2}_{\sym}$ and macroscopic material point $\mac\in\dom$ there exists a unique $\vartheta(s,G)\in \cV$ solving~\eqref{eq:corrector_equation}.
\end{lemma}

In \cite{boehnlein2023homogenized} we gave coordinate-free
definitions of $\Qhom$ and $\Beff$ that were formulated using orthonormal projections
on certain subspaces of the space of strains on the composite scale. For practical computations,
it is more convenient to work in a coordinate system of $\Rsym2$.
We therefore introduce an orthonormal basis $G_1,G_2,G_3$
of $\Rsym2$, and we write $\widehat A_1,\widehat A_2,\widehat A_3\in\R$ for the coefficients of
a matrix $A\in\Rsym2$ relative to that basis. We then set
$\vartheta_i(s)\colonequals\vartheta(s,G_i) \in \mathcal V$ for the correctors associated with the
basis elements $G_1, G_2, G_3$. Note that the solutions of~\eqref{eq:corrector_equation}
depend linearly on $G$, and therefore
\begin{equation*}
  \vartheta(s,G)=\sum_{i=1}^3\widehat G_i\vartheta_i(s)\qquad\text{for all }G\in\Rsym2.
\end{equation*}

The effective bending stiffness and prestrain are now constructed from the three
basis correctors $\vartheta_i$, $i=1,2,3$.

\begin{definition}[Effective bending stiffness $\Qhom$]\label{def:effectiveStiffness}
  For $\mac\in S$ let $\Qhomc(\mac)\in\R^{3\times 3}_{\sym}$ be given
  componentwise by
  \begin{align}
    \label{Qintegral}
    \Qhomc_{ij}(\mac)
    &\colonequals
    \int_{\Box} \Big\langle \mathbb L(\mac,y) \big(\iota(y_3 G_i)+ \corr_i(\mac,y) \big), \big(\iota(y_3 G_j)+\corr_j(\mac,y)   \big) \Big\rangle \dd y,
    \quad
    i,j=1,2,3.
  \end{align}
  We then define for all $G\in\R^{2\times 2}$
  \begin{align*}
    \Qhom(s,G) & \colonequals\sum_{i,j=1}^3\Qhomc_{ij}\widehat{\sym G}_i\widehat{\sym G}_j,
  \end{align*}
  where $\widehat{\sym G}_i$ is the $i$-th component of the symmetric part of $G$.
  The effective elasticity tensor $\mathbb L^\gamma_{\hom}(\mac)\in\operatorname{Lin}(\R^{2\times 2};\R^{2\times 2})$ is given by the identity
  \begin{equation*}
    \forall F,G\in\R^{2\times 2}\,:\,
    \big\langle \mathbb L^\gamma_{\hom}(s,y) F, G \rangle
    \colonequals
    \frac12\Big(Q^\gamma_{\hom}(s,F+G)-Q^\gamma_{\hom}(s,F)-Q^\gamma_{\hom}(s,G)\Big).
  \end{equation*}
\end{definition}

\begin{definition}[Effective prestrain $\Beff$]\label{def:effectivePrestrain}
  For $\mac\in S$ let $\Bc(\mac) \in \R^3$ be given componentwise by
  \begin{align*}
    \Bc_i(\mac)
    &\colonequals
    \int_{\Box} \Big\langle \mathbb L(\mac,y)\big(\iota(y_3 G_i)+ \corr_i(s,y) \big), B \Big\rangle\dd y,
    \qquad
    i=1,2,3.
  \end{align*}
  We set
  \begin{align*}
    \Beff(s)  & \colonequals\sum_{i=1}^3 \widehat{B}_{\textup{eff}, i}(s) G_i \in \Rsym2,
  \end{align*}
  where $\widehat{B}_{\textup{eff}}(s)\in\R^3$ is the unique solution to
  \begin{equation}
    \Qhomc(\mac) \Beffc(\mac) = \widehat{B}(\mac).
    \label{eq:prestrainSystemA}
  \end{equation}
\end{definition}

We recall some properties of $\Qhom$ and $\Beff$ from \cite{boehnlein2023homogenized}.
The quadratic form $\Qhom(s,\cdot)$ is positive definite and satisfies
\begin{equation*}
  \forall G\in\Rsym2\,:\,\frac{1}{c_1}\abs{G}^2\leq\Qhom(s,G)\leq c_1|G|^2,
\end{equation*}
where $c_1 > 0$ only depends on the constant $c_0$ from the ellipticity~\eqref{eq:Lbounds},
and is independent of $\mac\in S$.  The corresponding
matrix $\Qhomc(\mac)$ defined in~\eqref{Qintegral} is positive definite
and bounded (both uniformly in $\mac$). As a consequence the linear
system \eqref{eq:prestrainSystemA} admits a unique solution and thus $\Beff(\mac)$ is well-defined.
Standard energy estimates for $\vartheta_i(\mac)$ yield
\begin{equation*}
  |\Beff(\mac)|\leq c_2\|B(\mac,\cdot)\|_{L^2(\Box)},
\end{equation*}
for a constant $c_2$ only depending on $c_0$.

\subsection{Reformulation of the effective two-dimensional problem}

We now discuss a reformulation of the effective minimization Problem~\ref{prob:macminorig}
that seems more suitable for finite element discretizations.
As already observed by \cite{bartels2017bilayer,bartels2022stable},
for isometry deformations $\deform \in H^2_\iso$ one can bring the effective plate
energy $\energy$ defined in~\eqref{eq:Energy}
into an equivalent form where the second fundamental form is replaced by the Hessian $\nabla^2\deform$.
The latter is a third-order tensor with entries $(\nabla^2\deform)_{ijk} \colonequals \partial_j \partial_k \deform_i$.
For any vector $a \in \R^3$ and matrix $A \in \R^{2\times 2}$ let $a \otimes A$ be the
3-tensor with entries $(a\otimes A)_{ijk} \colonequals a_iA_{jk}$.
For a three-tensor $\mathbf{A}\in \R^{3\times 2\times 2}$ and $a\in\R^3$, write $a \cdot \mathbf{A}$ for the matrix in $\R^{2\times 2}$ with entries $(a\cdot \mathbf{A})_{jk}\colonequals\sum_{i=1}a_i\mathbf{A}_{ijk}$.

The reformulation is based on the following classical identities (see \cite[Proposition~3]{Muller2005}):
\begin{lemma}
  For any isometric immersion $\deform\in H^2_{\iso}(S;\R^3)$,
  \begin{equation}\label{geometric:identities}
    \II_\deform=-n_\deform\cdot \nabla^2\deform,
    \qquad \text{and} \qquad
    n_{\deform}\otimes(n_\deform\cdot \nabla^2\deform)=\nabla^2\deform,
  \end{equation}
  where $n_\deform\colonequals \partial_1\deform\times \partial_2\deform$ is the unit normal of $\deform$.
  In particular, $-n_\deform\otimes\II_\deform=\nabla^2\deform$.
\end{lemma}
These identities allow to replace $\II_\deform$ by $\nabla^2 \deform$ in bending models.
However, the presence of a prestrain term makes things more complicated.
First, to reformulate the problem in terms of the third-order tensor $\nabla^2\deform$
we introduce an effective linearized material for third-order tensors $\mathbf{A}\in\R^{3\times 2\times 2}$ by
\begin{equation*}
  \bQhom(s,\mathbf{A})
  \colonequals
  \sum_{i=1}^3\Qhom(s,e_i \cdot \mathbf{A}).
\end{equation*}
Again, there is an associated stiffness tensor defined by
\begin{equation*}
  \forall \mathbf{F}, \mathbf{G} \in\R^{3\times 2\times 2}\,:\,
  \big\langle \boldsymbol{L}^{\gamma}_{\hom}(s) \mathbf{F}, \mathbf{G} \big\rangle
  \colonequals
  \frac12\Big(\bQhom(s,\mathbf{F}+\mathbf{G})-\bQhom(s,\mathbf{F})-\bQhom(s,\mathbf{G})\Big).
\end{equation*}
With these new quantities we can multiply out the density quadratic density $\Qhom$ of $\energy$.
Note that the following lemma does not make use of the precise definition of $\Qhom$,
and would also work for other quadratic forms.

\begin{lemma}
  Let $\deform \in H^2_{\iso}(\dom;\R^3)$, 
  then
  \begin{align}
    \nonumber
    \Qhom \big(\mac , \II_\deform-\Beff(\mac)\big)
    & =
    \bQhom \big(\mac,\nabla^2\deform+n_\deform\otimes\Beff(\mac) \big)
    \\
    \label{eq:reformulation}
    & =
    \bQhom \big(\mac,\nabla^2\deform \big)+2\big\langle\boldsymbol{L}^\gamma_{\hom}(\mac)\nabla^2\deform, n_\deform\otimes\Beff(\mac)\big\rangle+\Qhom(\mac,\Beff(\mac)).
  \end{align}
\end{lemma}
\begin{proof}
  In view of the definition of $\bQhom$, for any $A\in\Rsym2$ and $a\in\R^3$ with $|a|_2^2=1$, we have
  \begin{equation*}
    \Qhom(\mac, A)
    =
    \sum_{i=1}^3(a\cdot e_i)^2\Qhom(\mac,A)
    =
    \sum_{i=1}^3\Qhom(\mac,e_i\cdot(a\otimes A))
    =
    \bQhom(\mac, a\otimes A).
  \end{equation*}
  Use this with $a = n_\deform$ and $A = \II_\deform-\Beff(\mac)$ to get
  \begin{align*}
    \Qhom\left(\mac,\II_\deform-\Beff(\mac)\right)
    =
    \bQhom\left(\mac,n_\deform\otimes\big(\II_\deform-\Beff(\mac)\big)\right).
  \end{align*}
  Combined with \eqref{geometric:identities} we obtain
  \begin{align*}
    \Qhom\left(\mac,\II_\deform-\Beff(\mac)\right)
    =
    \bQhom\big(\mac,\nabla^2\deform+n_\deform\otimes\Beff(\mac)\big),
  \end{align*}
  which is the first equality.
  By expanding the quadratic term on the right, the second identity follows.
\end{proof}

Since the last term in~\eqref{eq:reformulation} is independent of the deformation $\deform$, we can subtract it from $\energy$ without changing the set of minimizers
or the set of critical points, to get
\begin{equation}
  \label{eq:reformulated_energy}
  \tenergy
  \colonequals
  \int_S\bQhom(\mac,\nabla^2\deform)+2\Big\langle\boldsymbol{L}^\gamma_{\hom}(\mac)\nabla^2\deform,n_\deform\otimes\Beff(\mac)\Big\rangle \dd \mac.
\end{equation}
We arrive at the following minimization problem, which is equivalent to the
effective Problem~\ref{prob:macminorig}:
\begin{problem}[Reformulated problem]
  \begin{equation*}
    \text{Minimizer $\tenergy(\deform)$ subject to $\deform\in\cA$}.
  \end{equation*}
\end{problem}

This reformulation has several advantages:

\begin{itemize}
  \item First of all, it
  replaces the nonlinear second fundamental form $\II_\deform$ form by the Hessian $\nabla^2\deform$,
  which is a linear expression. The price to pay is the appearance of
  the nonlinearity $n_\deform\otimes\Beff$, which is however of lower order.
  Works like~\cite{bartels2022stable} have exploited this by treating the
  (linear) first part of~\eqref{eq:reformulated_energy} implicitly in a gradient flow algorithm, and the (nonlinear)
  second part explicitly.
  
  \item Since $\II_{jk} \colonequals \partial_jv\cdot\partial_k n$
  is cubic in the derivatives of $\deform$ but $\nabla^2v$ is only linear,
  such a reformulation lowers the degree of the integrand if $H^2_\text{iso}$ is approximated
  by a space of piecewise polynomials.
  
  \item Discretizing $\II_{\deform}$ in a way that preserves coercivity is delicate. 
  The coercivity argument in the continuum uses two crucial structural facts about the 
  second fundamental form: (i) $\II_{\deform}$ is symmetric, and (ii) 
  $|\II_{\deform}|^2 = |\nabla^2 \deform|^2$. In the proof of 
  Theorem~\ref{T:main} we consider a natural discretization of $\II_\deform$, namely $n_H \cdot \nabla\nabla_H \deform$, where $n_H$ is defined by interpolating the normal to the discrete partial derivatives at the nodes.  We prove that it converges to
  $\II_{\deform}$. However, this discretization does not provide the coercivity, because boundedness of 
  $\sym(n_H \cdot \nabla\nabla_H \deform)$ in $L^2$ does not imply boundedness of $\nabla\nabla_H\deform$.
  In contrast, the reformulated energy yields for each component $i=1,2,3$ control of $\sym\nabla\nabla_H\deform_i$ in $L^2$, and thus the coercivity is obtained by Korn's inequality.
\end{itemize}

Finally, we note that in actual numerical experiments the reformulated functional
appears to perform much better than the original one. However, the exact reason for this
remains unclear.

\section{Numerical approximation of the effective quantities}
\label{sec:MicroNumerics}
In this section, we introduce numerical approximations $\Qhomh(\mac,\cdot)$
and $\Beffh(\mac)$ of the effective quantities $\Qhom(\mac,\cdot)$ and $\Beff(\mac)$,
respectively. These approximations are used in the global discretization of
the energy functional $\tenergy$ presented in Chapter~\ref{sec:MacroDiscretization},
where they are evaluated at the points
of a quadrature rule for the integral $\int_S\dd s$.
Their definitions are based on a
finite element approximation of the corrector problem \eqref{eq:corrector_equation},
which yields approximate correctors $\vartheta^{h}(\mac,G)$.
The approximate quantities $\Qhomh(\mac,\cdot)$ and $\Beffh(\mac)$ are then defined 
analogously to the continuum setting (Definitions~\ref{def:effectiveStiffness} and~\ref{def:effectivePrestrain}),
with the difference that the integral $\int_\Box \,\dd y$ is replaced by a
numerical quadrature operator $\cG_{\Box}^{h}$.

As we shall explain below, under appropriate assumptions we obtain
\begin{align*}
  \norm{\vartheta^{h}(\mac) - \vartheta(\mac)}_{H^1(\Box;\R^3)} \to 0
\end{align*}
by a $\Gamma$-convergence argument,
and, as a corollary, the convergence of $\Qhomh$ and $\Beffh$ a.e.\ in~$S$.
This is less than the corresponding result in~\cite{rumpf2023two}, which shows even
convergence with a linear rate. This of course requires
stronger regularity assumptions on the coefficients~$\mathbb L$ than we make here,
and rules out piecewise constant composites. We settle for rate-less convergence
under weaker assumptions here, because our convergence proof for the overall discretization
(Theorem~\ref{T:main} below) does not need a rate for the convergence of
$\Qhomh(\mac,\cdot)$ and $\Beffh(\mac)$ as $h\to0$.

\subsection{Discretization of the corrector problem}
\label{sec:discreteCorr}

We now define the discretization of the corrector problem~\eqref{eq:corrector_equation}
using standard Lagrange finite elements.
Let $\microgrid$ be a shape-regular family of partitions of the domain $\Box$
into axis-aligned hexahedral elements $K\in \microgrid$, parametrized by a
microscopic discretization parameter $h$ representing the maximal edge length.
Note that not all elements have to have the same size, see Figure~\ref{fig:RVEmesh}.
Note that energy density and prestrain can vary smoothly inside the mesh cells $K\in\microgrid$.
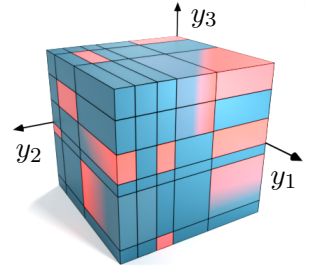
\begin{wrapfigure}[20]{R}{4.0cm} 
  \label{fig:RVEmesh}
  \centering
  \begin{tikzpicture}
    \node[anchor=south west, rotate=0] at (0,0,0) {\includegraphics[width=0.25\textwidth]{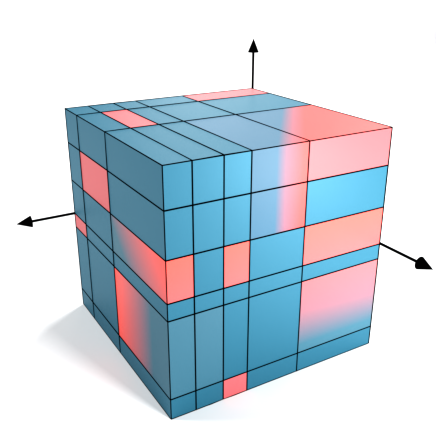}};;
    \node at (0.5,1.75,0) {$y_2$};
    \node at (3.85,1.4,0) {$y_1$};
    \node at (2.8,3.55,0) {$y_3$};
  \end{tikzpicture}
  \caption{Example of an admissible hexahedral partition $\microgrid$.
  The mesh resolves jumps in the coefficients. The coefficients can vary
  within each element, but have to remain Lip\-schitz continuous.}
\end{wrapfigure}
Denote the space of multilinear functions on $K\in\microgrid$ by
$\cQ_1(K)\colonequals\big\{f \in L^2(K)\,:\,f(y)=\prod_{i=1}^3(a_i+b_iy_i)\text{ with $a_i,b_i\in\R$} \big\}$,
and let
\begin{align*}
  \cQ_1
  \colonequals
  \left\{ v\in C^0(\Box;\R^3) \ : {v|}_K \in [\cQ_1(K)]^3 \text{ for all $K \in \microgrid$} \right\}
\end{align*}
be the space of continuous piecewise $\cQ_1(K)$ vector fields on $\Box$.
Denote its subspace of in-plane periodic vector fields with zero integral by
\begin{align*}
  \cQ_{1}^{\text{per}}\colonequals \cQ_1 \cap  H^1_\per(\Box;\R^3).
\end{align*}
The corrector finite element space $\cV^h \subset \cV$ consists of sums of affine deformations
and in-plane-periodic finite element vector fields
\begin{equation*}
  \mathcal{V}^h
  \colonequals
  \Big\{ \vartheta_h \in C^0(\Box;\R^3) : \vartheta_h(y) = \iota(M)y + \varphi_h(y), \ M \in \Rsym2,\  \varphi_h \in \cQ_1^{\text{per}} \Big\}.
\end{equation*}
This is evidently a subspace of the corrector space~$\cV$ of Section~\ref{sec:homogenization_formulas}.

To approximate integrals, for every $K\in \microgrid$ we consider a quadrature scheme
consisting of a set of quadrature points $r^K_j\in K$ and weights $\omega_j^K>0 $
for $j=1,\ldots,l$. We denote the set of all quadrature points on the grid $\microgrid$ by
\begin{equation*}
  \mathcal R_h(\Box)\colonequals \big\{r^K_j\,:\,K\in\microgrid,\,j=1,\ldots,n\big\},
\end{equation*}
and define the quadrature operators
\begin{alignat*}{2}
  \cG_K^h & : C^0(K;\R) \to \R, & \qquad       \cG_K^h[f] &\colonequals  \sum_{j=1}^{l} \omega_j^K f(r^K_j),
  \\
  \cG_\Box^h & : C^0(\Box;\R) \to \R, &     \cG_\Box^h[f] &\colonequals  \sum_{K\in \microgrid}\cG_K^h[f].
\end{alignat*}

Since we need our estimates to be uniform in $\mac\in S$ and $h > 0$, we assume that
the quadrature schemes for the $K\in\microgrid$ are all created from a single reference quadrature rule
$(\hat{r}_j,\hat{\omega}_j)$ on the reference element $\hat{K}\colonequals[0,1]^3$
by translation and dilation. We make the following assumptions (see Appendix~\ref{sec:appendixA} for details):
\begin{enumerate}[label=(\alph*)]
  \item The set $\{\hat{r}_{j}\}_{j=1,\ldots,l}$  contains a $\cQ_1(\hat{K})$-unisolvent subset.
  \item The quadrature rule $(\hat{r}_j,\hat{\omega}_j)$ is exact on $\cQ_1(\hat{K})$.
\end{enumerate}

The approximate corrector $\vartheta^h(\mac,G)$ can then be defined as the unique solution
in $\mathcal{V}^h$ of the linear system
\begin{equation}\label{eq:corrector_discrete}
  \cG_{\Box}^{h}\Big[\big \langle\mathbb{L}(\mac,\cdot)(\iota(y_3G)+\nabla_\gamma\vartheta^{h}(\mac,\cdot)),\,\nabla_\gamma \psi^h(\cdot) \big\rangle\Big]
  =0
  \qquad
  \text{for all $\psi^h\in\mathcal{V}^h$}.
\end{equation}
As previously we define $\vartheta_i^h(\mac) \colonequals \vartheta^h(\mac,G_i)$,
where $G_1, G_2, G_3$ is the chosen orthonormal basis of $\Rsym2$.

The quadrature operator $\cG^h_\Box$ invokes the evaluation of $\mathbb L(\mac,y)$ and $B(\mac,\cdot)$
at all points from the quadrature point set $R_h(\Box)$. 
For the time being we simply assume that this evaluation is well defined.
Later, in Proposition~\ref{prop:approxcor}, we will make a local regularity assumption
that enforces $\mathbb L(\mac,\cdot)$ and $B(\mac,\cdot)$ to have representatives that are continuous at all quadrature points $y\in R_h(\Box)$.

\subsection{Analysis of the approximate corrector problem}

\begin{lemma}[Existence and energy estimate]\label{L:approx:corr}
  Let $\mac\in S$ and assume that $\mathbb L(\mac,y)$ satisfies the ellipticity bound \eqref{eq:Lbounds} for all $y\in R_h(\Box)$. Then for every $G \in \Rsym2$ the corrector $\vartheta^h(\mac,G)$ exists,
  is unique and depends continuously on $G$
  \begin{equation}
    \norm{\vartheta^h(\mac,G)}_{H^1(\Box)}\leq C \abs{\sym G}, \label{eq:energyEstimate}
  \end{equation}
  where $C$ only depends on the constant $c_0$ of~\eqref{eq:Lbounds},
  the reference quadrature rule, and the limit $\gamma$.
\end{lemma}
\begin{proof}
  For convenience introduce the space
  \begin{equation*}
    \cQ_1^\text{dG}(\Box; \R^{3 \times 3})
    \colonequals
    \Big\{\Phi\in L^2(\Box;\R^{3\times 3})\,:\,\Phi\vert_K\in [\cQ_1(K)]^{3 \times 3}\text{ for all $K\in\microgrid$}\Big\},
  \end{equation*}
  and the bilinear form
  \begin{align*}
    a^h(\mac,\cdot,\cdot) : \cQ_1^\text{dG}(\Box; \R^{3 \times 3}) \times \cQ_1^\text{dG}(\Box; \R^{3 \times 3}) \to\R,
    \qquad
    a^h(\mac,\Phi,\Psi)\colonequals \cG^h_{\Box}\big[\langle\mathbb L(\mac,\cdot)\Phi,\Psi\rangle\big].
  \end{align*}
  In Appendix~\ref{sec:appendixA} it is shown that by the unisolvency of the quadrature rule
  on $\cQ_1$, we have that
  \begin{equation*}
    \forall g\in\cQ_1(K)\,:
    \qquad
    \widehat C^{-1}\|g\|_{L^2(K)}\leq \cG_{K}[|g|^2]^\frac12\leq \widehat C\|g\|_{L^2(K)},
  \end{equation*}
  for a $\widehat C$ that can be chosen independently of $K\in\microgrid$ and the mesh size $h$.
  Using then the ellipticity and boundedness of $\mathbb L(\mac,\cdot)$, we deduce that
  \begin{align*}
    \frac{1}{\tilde C}\|\sym\Phi\|^2_{L^2(\Box)}\leq a^h(\mac,\Phi,\Phi),
    \qquad \text{and} \qquad
    |a^h(\mac,\Phi,\Psi)|\leq \tilde C\|\sym\Phi\|_{L^2(\Box)}\|\sym\Psi\|_{L^2(\Box)},
  \end{align*}
  where $\tilde C$ only depends on the reference quadrature rule and $c_0$ from \eqref{eq:Lbounds}.
  We further see that for any $G\in\R^{2\times 2}$ and $\vartheta^h\in\mathcal V^h$,
  we have $\iota(y_3G) \in \cQ_1^\text{dG}(\Box; \R^{3 \times 3})$ and $\nabla_\gamma\vartheta^h\in \cQ_1^\text{dG}(\Box; \R^{3 \times 3})$, and
  \begin{equation*}
    \frac{1}{C_\gamma}(|\sym G|^2+\|\vartheta^h\|_{H^1(\Box)}^2)\leq \|\sym(\iota(y_3G)+\nabla_\gamma\vartheta^h)\|_{L^2(\Box)}^2\leq C_\gamma(|\sym G|^2+\|\vartheta^h\|^2_{H^1(\Box)}),
  \end{equation*}
  where $C_{\gamma}$ only depends on $\gamma$. 
  We note that the corrector problem \eqref{eq:corrector_discrete} for $\vartheta^h(\mac,G)$ takes the form
  \begin{equation*}
    a^h(\mac,\nabla_\gamma\vartheta^h(\mac,G),\nabla_\gamma\psi^h)=-a^h(\mac,\iota(y_3G),\nabla_\gamma\psi^h)
    \qquad
    \text{for all $\psi^h\in\mathcal{V}^h$}.
  \end{equation*}
  In view of the estimates derived for the bilinear form $a^h$ and the boundedness of the right-hand side, we see that the existence, uniqueness and the estimate \eqref{eq:energyEstimate} follow from the Lax--Milgram theorem.
\end{proof}

With the help of the approximate correctors $\vartheta^h_i \colonequals \vartheta^h(s,G_i)$,
we define the approximate effective
bending stiffness $\Qhomh$ and the approximate effective prestrain $\Beffh$
in analogy to Definitions~\ref{def:effectiveStiffness} and~\ref{def:effectivePrestrain}.

\begin{definition}[Approximate effective bending stiffness $\Qhomh$]
  Let $\mac\in S$, and set $\widehat Q^h(\mac)\in\R^{3\times 3}_{\sym}$
  componentwise by
  \begin{align*}
    \widehat Q_{ij}^h(\mac)
    &\colonequals 
    \cG_\Box^h\Big[ \Big \langle\mathbb{L}(\mac,y)\big(\iota(y_3 G_i) + \nabla_\gamma \vartheta_{i}^h(y)\big),\,\big(\iota(y_3 G_j) + \nabla_\gamma \vartheta_{j}^h(y)\big) \Big \rangle\Big],
    \quad
    \forall i,j=1,2,3.
  \end{align*}
  Then define for all $G\in\R^{2\times 2}$
  \begin{align*}
    \Qhomh(s,G)
    \colonequals
    \sum_{i,j=1}^3\widehat Q_{ij}^h\widehat{\sym G_i}\widehat{\sym G_j}.
  \end{align*}
  The homogenized stiffness tensor $\mathbb L^{\gamma,h}_{\hom}(\mac)\in\operatorname{Lin}(\R^{2\times 2};\R^{2\times 2})$ is given by the identity
  \begin{equation*}
    \forall F,G\in\R^{2\times 2}\,:\quad
    \big\langle \mathbb L^{\gamma,h}_{\hom}(s) F, G \big\rangle
    \colonequals
    \frac12\Big(Q^{\gamma,h}_{\hom}(s,F+G)-Q^{\gamma,h}_{\hom}(s,F)-Q^{\gamma,h}_{\hom}(s,G)\Big).
  \end{equation*}
\end{definition}

Note that
\begin{equation*}
  \big\langle\mathbb L^{\gamma,h}_{\hom}(\mac)F,G \big\rangle
  =
  \Big\langle\mathbb L^{\gamma,h}_{\hom}(\mac)\sym F,\sym G \Big\rangle
\end{equation*}
for all $F, G \in \R^{2 \times 2}$, as for the original three-dimensional stiffness $\mathbb L$.

\begin{definition}[Approximate effective prestrain $\Beffh$]
  Let $\mac\in S$. We define $\widehat B^h(\mac)\in \R^3$ componentwise by
  \begin{align*}
    \widehat B_{i}^h(\mac)
    \colonequals \cG_\Box^h \Big[\Big\langle\mathbb{L}(\mac,y)\big(\iota(y_3G_i) + \nabla_\gamma\vartheta_{i}^h(y)\big),\,B(\mac,y) \Big\rangle\Big],
    \qquad
    i=1,2,3.
  \end{align*}
  We then set
  \begin{equation*}
    \Beffh(s)\colonequals\sum_{i=1}^3 \widehat{B}_{\textup{eff}, i}^h(s) G_i \in \Rsym2,
  \end{equation*}
  where $\widehat{B}_{\textup{eff}}^h(s)\in\R^3$ is the unique solution to
  \begin{equation*}
    \Qhomc(\mac) \widehat{B}_{\textup{eff}}^h(\mac) = \widehat{B}^h(\mac).
  \end{equation*}
\end{definition}

\begin{remark}[Variational characterization of the corrector]\label{rem:variational_characterization}
  The effective bending stiffness and its approximation can be equivalently characterized by the following minimization problems:
  \begin{align}
    \label{min:cont}
    Q^{\gamma}_{\hom}(\mac, G) & = \min_{\vartheta\in\mathcal V}
    \int_\Box Q(s,y,\iota(y_3G)+\nabla_\gamma\vartheta)\,\dd y,
    \\
    \label{min:discrete}
    Q^{\gamma,h}_{\hom}(\mac, G) & = \min_{\vartheta^h\in\mathcal V^h}
    \cG_\Box^h\Big[Q(s,y,\iota(y_3G)+\nabla_\gamma\vartheta^h)\Big].
  \end{align}
  Indeed, this follows since the minimization problems are convex and coercive,
  and the Euler--Lagrange equations are exactly \eqref{eq:corrector_equation} and \eqref{eq:corrector_discrete}, respectively.
  In particular, the minimizers of \eqref{min:cont} and \eqref{min:discrete} are the correctors $\vartheta(\mac,G)$ and $\vartheta^h(\mac,G)$, respectively.
\end{remark}

The approximate effective quantities inherit the ellipticity properties
of their three-dimensional counterparts.

\begin{lemma}[Boundedness and ellipticity]
  Let $\mac\in S$ and assume that $\mathbb L(\mac,y)$ satisfies
  the ellipticity bound \eqref{eq:Lbounds} for all $y\in R_h(\Box)$.
  Then for all $F, G \in \R^{2 \times 2}$ we have
  \begin{align}
    \nonumber
    \langle\mathbb L^{\gamma,h}_{\hom}(\mac)F,G\rangle
    & \leq
    C|\sym F||\sym G|,
    \\
    \label{eq:Lhom_ellipticity}
    \langle\mathbb L^{\gamma,h}_{\hom}G,G\rangle
    & \geq
    \frac1{C}|\sym G|^2,
    \\
    \nonumber
    |B^{\gamma,h}_{\rm eff}(\mac)|
    & \leq
    C\max_{y\in R_h(\Box)}|B(\mac,y)|,
  \end{align}
  where $C>0$ denotes a constant that only depends on the ellipticity constant $c_0$ from \eqref{eq:Lbounds}, the reference quadrature rule, and the limit $\gamma$.
\end{lemma}
\begin{proof}
  The estimates follow from the corresponding bounds \eqref{eq:Lbounds} by a standard argument.
  We only prove the lower bound~\eqref{eq:Lhom_ellipticity}.
  Indeed, for $G\in\R^{2\times 2}$ and $\vartheta^h\in\mathcal V^h$
  set $F \colonequals \iota(y_3 G)+\nabla_\gamma\vartheta^h$, and note that $F$ and $\sym F$ belong to
  the space $\cQ_1^\text{dG}(\Box;\R^{3 \times 3})$ of discontinuous $\cQ_1$ matrix fields.
  Thus, by \eqref{eq:Lbounds} and the quadrature error bound \eqref{eq:equivnorm} we find that
  \begin{equation*}
    \cG_\Box^h\big[\langle \mathbb L(\mac)F,\,F\big\rangle]
    \geq
    \frac{1}{c_0}\cG_\Box^h\Big[\langle \sym F,\,\sym F\big\rangle\Big]
    \geq \frac{\tilde C}{c_0}\int_{\Box}|\sym F|^2\,\dd y,
  \end{equation*}
  where $\tilde C$ is a constant only depending on the reference quadrature rule.
  Multiplying out the expression $\abs{\sym F}^2 = \abs{\iota(y_3 G)+\nabla_\gamma\vartheta^h}^2$,
  and noting that the mixed terms vanish because $\sym G$ and $\sym\nabla_\gamma\vartheta^h$
  are orthogonal in $L^2(\Box)$,
  we find that $\int_{\Box}|\sym F|^2\,\dd y=|\sym G|^2+\int_{\Box}|\sym\nabla_\gamma\vartheta^h|^2\,\dd y$,
  and thus, in particular, $\cG_\Box^h\big[\langle \mathbb L(\mac)F,\,F\big\rangle]\geq\frac{\tilde C}{c_0}|\sym G|^2$.
  Next, using the minimization formulation \eqref{min:discrete} of $\Qhomh$, by infimizing over $\vartheta^h$
  we obtain $\Qhomh(\mac,G)\geq\frac{\tilde C}{c_0}|\sym G|^2$ and thus the claimed
  lower-bound follows with $C \colonequals \frac{\tilde C}{c_0}$.
\end{proof}

It is obvious that the constant $C$ of the previous lemma depends on the constant $c_0$.
However, it is less obvious that it also depends on the reference quadrature rule (a minor fact that is not mentioned in \cite{rumpf2023two}).
We note that if the reference quadrature rule was exact for
$\cQ_2(\hat K)$, then we could choose $\tilde C=1$.

\medskip

For the convergence as $h\to 0$ we require $\mathbb L$ and $B$ to be locally Lipschitz on the elements $K\in\microgrid$. 

\begin{proposition}[Convergence of the micro-discretization]\label{prop:approxcor}
  Let $\mac\in S$. Assume that $\mathbb L(\mac,\cdot)$ satisfies the ellipticity bound \eqref{eq:Lbounds}, and assume local Lipschitz regularity of the oscillations of $\mathbb L(s,\cdot)$ and $B(s,\cdot)$ with respect to the elements of $\microgrid$ in the sense that
  \begin{equation*}
    \lim_{h\to 0}h\max_{K\in\microgrid}\Big(\|\nabla_y\mathbb L(\mac,\cdot)\|_{L^\infty(K)}\big(1+\|B(\mac,\cdot)\|_{W^{1,\infty}(K)}\big)\Big)=0.
  \end{equation*}
  Then, as $h \to 0$,
  \begin{equation*}
    \Qhomh(s,\cdot) \to \Qhom(s,\cdot)
    \qquad \text{and} \qquad
    \Beffh(\mac)\to\Beff(\mac).
  \end{equation*}
  In particular,
  \begin{equation*}
    \mathbb L^{\gamma,h}_{\hom}(\mac)\to \mathbb L^{\gamma}_{\hom}(\mac)
    \qquad \text{and} \qquad
    \|\vartheta^h_i(\mac)-\vartheta_i(\mac)\|_{H^1(\Box)}\to 0,
    \quad \forall i=1,2,3.
  \end{equation*}
\end{proposition}
\begin{proof}
  Throughout the proof $C$ denotes a constant that can change from line to line,
  but which can be chosen independent of $h$ and $\mac\in\dom$.
  
  \begin{enumerate}[wide,label=\textbf{Step~\arabic*}, labelindent=0pt]
    \item \label{proofstep:quadrature_errors} (Quadrature estimate)
    
    For convenience define the general quadrature error operators
    $E^h_{\Box}[f]\colonequals\cG^h_\Box[f]-\int_\Box f\dd y$ and $E^h_{K}[f]\colonequals\cG^h_K[f]-\int_K f \dd y$.
    We claim that for any $G\in\R^{2\times 2}$ and any $u^h\in\mathcal V^h$ we have
    \begin{multline*}
      \abs[\Big]{E^h_{\Box}\Big[
      \Big\langle\mathbb{L}(\mac,y)\big(\iota(y_3 G) + \nabla_\gamma u^h(y)\big),\,\big(\iota(y_3 G) + \nabla_\gamma u^h(y)\big)\Big\rangle
      \Big]}
      \\
      \leq Ch\sup_{K\in\microgrid}\|\nabla_y\mathbb L\|_{\infty}(|G|^2+\|u^h\|^2_{H^1(\Box)}).
    \end{multline*}
    Since $E_\Box^h$ is linear we can multiply out the scalar product on the left
    to obtain three terms. With the triangle inequality it then suffices to show that
    \begin{align}
      \label{prop:approxcor:qe1}
      \Big|E^h_K\Big[\langle\mathbb L(s,\cdot)\iota(y_3G),\,\iota(y_3G)\rangle\Big]\Big|
      & \leq Ch\|\nabla_y\mathbb L\|_{L^\infty(K)}|K||G|^2,\\
      \label{prop:approxcor:qe2}
      \Big|E^h_K[\langle\mathbb L(s,\cdot)\nabla_\gamma u^h,\,\iota(y_3G)\rangle]\Big|
      & \leq Ch\|\nabla_y\mathbb L\|_{L^\infty(K)}|K|^\frac12|G|\|u^h\|_{H^1(K)},\\
      \label{prop:approxcor:qe3}
      \Big|E^h_K[\langle\mathbb L(s,\cdot)\nabla_\gamma u^h,\,\nabla_\gamma u^h\rangle]
      & \leq Ch\|\nabla_y\mathbb L\|_{L^\infty(K)}\|u^h\|^2_{H^1(K)}.
    \end{align}  
    These estimates follow from the quadrature error estimates given in Appendix~\ref{sec:appendixA}:
    For \eqref{prop:approxcor:qe1}, use \eqref{eq:quadest2} with $f=\mathbb L(s,\cdot)\iota(y_3G)$ and $g^h=\iota(y_3G)$.
    Estimate \eqref{prop:approxcor:qe2} follows from \eqref{eq:quadest0} with $f=\mathbb L(s,\cdot)^\top\iota(y_3G)$, $g^h=u^h$, and $\hat g^h=y_j$, $j=1,2$ (so that $\partial_j\hat g_j=1$).
    Finally, for \eqref{prop:approxcor:qe3} also use \eqref{eq:quadest0}, this time with $f=\mathbb L(s,\cdot)$, $g^h=\hat g^h=u^h$.
    
    \item (Convergence of the bending stiffness)\label{proofstep:convergenceQhom}
    
    We then show that for all $G\in\R^{2\times 2}$ we have
    \begin{equation*}
      \lim_{h\to 0}Q^{\gamma,h}_{\hom}(\mac, G)=Q^{\gamma}_{\hom}(\mac, G).
    \end{equation*}
    For the argument let $\vartheta_G\colonequals\vartheta(\mac, G)$
    and $\vartheta_G^h\in\mathcal V^h$ denote the exact and approximate correctors at~$\mac$
    associated with $G$, respectively.
    As argued in Remark~\ref{rem:variational_characterization},
    the corrector $\vartheta_G$ minimizes the functional
    \begin{equation}\label{eq:contfunct}
      \mathcal V\ni u\mapsto \int_\Box Q\big(s,y,\iota(y_3G)+\nabla_\gamma u\big)\dd y.
    \end{equation}
    Likewise, the approximate corrector $\vartheta^h_G$ minimizes the discrete functional
    \begin{equation*}
      \mathcal V^h\ni u^h\mapsto \cG^h_{\Box}\Big[ Q\big(s,y,\iota(y_3G)+\nabla_\gamma u^h\big)\Big].
    \end{equation*}
    We thus conclude that
    \begin{align}
      \label{eq:qhom_definition_recall}
      Q^{\gamma}_{\hom}(s,G)
      & = \int_\Box Q\big(s,y,\iota(y_3G)+\nabla_\gamma\vartheta_G \big)\dd y\\
      \label{eq:corrector_convergence}
      & \leq \int_\Box Q \big(s,y,\iota(y_3G)+\nabla_\gamma\vartheta_G^h \big)\dd y\\
      \label{eq:qhomh_convergence}
      & = \cG^h_{\Box}\Big[ \Big\langle \mathbb L(\mac,y)\big(\iota(y_3G)+\nabla_\gamma\vartheta_G^h\big),\,\big(\iota(y_3G)+\nabla_\gamma\vartheta_G^h\big)\Big \rangle\Big]-e^h,
    \end{align}
    where
    \begin{equation*}
      e^h
      \colonequals
      E^h_{\Box}\Big[\Big\langle \mathbb L(\mac,y)\big(\iota(y_3G)+\nabla_\gamma\vartheta_G^h\big),\,\big(\iota(y_3G)+\nabla_\gamma\vartheta_G^h\big)\Big \rangle\Big].
    \end{equation*}
    Then, choose a sequence $\widetilde{\vartheta}^h\in\mathcal V^h$ such that $\widetilde{\vartheta}^h\to \vartheta_G$ strongly in $H^1(\Box)$ as $h\to 0$.
    By appealing to the minimality of $\vartheta^h_G$, we get
    \begin{align}
      \nonumber
      Q^{\gamma}_{\hom}(s,G)
      & \leq \cG^h_{\Box}\Big[\Big\langle \mathbb L(\mac,y)\big(\iota(y_3G)+\nabla_\gamma\widetilde\vartheta^h\big),\,\big(\iota(y_3G)+\nabla_\gamma\widetilde\vartheta^h\big)\Big\rangle\Big]-e^h
      \\
      \label{eq:qhomh_bound_two_quad_errors}
      & = \int_\Box\Big\langle\mathbb L(\mac,y)(\iota(y_3G)+\nabla_\gamma\widetilde\vartheta^h),(\iota(y_3G)+\nabla_\gamma\widetilde\vartheta^h)\Big\rangle \dd y-e^h+\widetilde e^h,
    \end{align}
    where
    \begin{equation*}
      \widetilde e^h
      \colonequals
      E^h_{\Box}\Big[\Big\langle \mathbb L(\mac,y)\big(\iota(y_3G)+\nabla_\gamma\widetilde\vartheta^h\big),\,\big(\iota(y_3G)+\nabla_\gamma\widetilde\vartheta^h\big)\Big\rangle\Big].
    \end{equation*}
    Now, since $\widetilde\vartheta^h\to \vartheta_G$ strongly in $H^1(\Box)$,
    and because the quadrature errors  $e^h$ and $\widetilde e^h$ vanish thanks to~\ref{proofstep:quadrature_errors},
    we deduce that~\eqref{eq:qhomh_bound_two_quad_errors} converges to $Q^\gamma_{\hom}(s,G)$
    as $h \to 0$, and thus we conclude that all terms between \eqref{eq:qhom_definition_recall} and \eqref{eq:qhomh_bound_two_quad_errors}
    converge to $Q^\gamma_{\hom}(s,G)$ as $h\to 0$.
    In particular, this includes the numerical integral in \eqref{eq:qhomh_convergence},
    which is precisely $Q^{\gamma,h}_{\hom}(s,G)$. Hence, we conclude that $Q^{\gamma,h}_{\hom}(s,G)\to Q^{\gamma}_{\hom}(s,G)$ for all $G\in\R^{2\times 2}$ as claimed.
    
    As the terms $\mathbb L^{\gamma,h}_{\hom}(\mac)$ and $\mathbb L^\gamma_{\hom}(\mac)$
    depend continuously on $Q^{\gamma,h}_{\hom}(s,G)$ and $Q^{\gamma}_{\hom}(s,G)$,
    respectively, we conclude that $\mathbb L^{\gamma,h}_{\hom}(\mac)\to \mathbb L^\gamma_{\hom}(\mac)$
    as well.
    To show convergence of the correctors $\vartheta_i^h$, $i=1,2,3$,
    from the convergence of the integral~\eqref{eq:corrector_convergence} we conclude
    that $\vartheta^h_G$ is a minimizing sequence for the continuum functional \eqref{eq:contfunct}.
    Since the latter is uniformly convex in $\mathcal V$ we conclude that $\vartheta^h_G\to\vartheta_G$
    strongly in $H^1(\Box)$ for any $G \in \R^{2\times 2}_{\sym}$,
    in particular for the basis elements $G = G_1, G_2, G_3$.
    
    \item (Convergence of the effective prestrain)
    
    It remains to prove the convergence of $B^{\gamma,h}_{\rm eff}(s)$. To this end
    we first prove convergence of the prestrain coefficients $B^{\gamma,h}_{\textup{eff},i}$
    for $i=1,2,3$.
    The quadrature error results of Appendix~\ref{sec:appendixA} allow to bound
    \begin{equation*}
      \abs[\Big]{E^h_\Box\Big[\Big \langle\mathbb L(s,y)\big(\iota(y_3G_i)+\nabla_\gamma\vartheta_i^h\big),\,B(s,y)\Big\rangle\Big]}
      \leq
      Ch\max_{K\in\microgrid}\|\mathbb L(s,\cdot)\|_{W^{1,\infty}(K)}\big(1+\|B(s,\cdot)\|_{W^{1,\infty}(K)}\big).
    \end{equation*}
    This follows again by multiplying out the left-hand side, and then applying once
    \eqref{eq:quadest2} with $f=\mathbb L^\top(s,\cdot) B(s,\cdot)$, $g^h=\iota(y_3 G_i)$,
    and once \eqref{eq:quadest0} with $f=\mathbb L^\top(s,\cdot) B(s,\cdot)$, $g^h=\vartheta_i^h$, $\hat g^h=y_j$, $j=1,2$, combined with the continuity estimate~\eqref{eq:energyEstimate} for $\vartheta^h_i$.
    
    Since the right-hand side of this quadrature estimate vanishes for $h\to 0$
    by assumption on the regularity of $\mathbb L$ and $B$, and since the constant
    is independent of $h$, we conclude that
    \begin{align*}
      \lim_{h\to 0}B^{\gamma,h}_{\textup{eff},i}(s)
      & =
      \lim_{h\to 0}\int_\Box \Big\langle\mathbb L(s,y)\big(\iota(y_3G_i)+\nabla_\gamma\vartheta_i^h\big),\,B(s,y)\Big\rangle \dd y
      \\
      & =
      \int_\Box\Big\langle\mathbb L(s,y)\big(\iota(y_3G_i)+\nabla_\gamma\vartheta_i\big),\,B(s,y)\Big\rangle \dd y
      \\
      & =
      B^\gamma_{\textup{eff},i}(s),
    \end{align*}
    where we used the strong convergence of $\nabla_\gamma\vartheta^h_i\to \nabla_\gamma\vartheta_i$ in $L^2(\Box)$, which we have already established in \ref{proofstep:convergenceQhom}.
    Finally, of course, the convergence of the coefficients $B^{\gamma,h}_{\textup{eff},i}$
    implies convergence of $B^{\gamma,h}_{\textup{eff}}$ itself.
    \qedhere  
  \end{enumerate}
\end{proof}

\begin{remark}[Quantitative convergence]
  In \cite{rumpf2023two} the authors proved a stronger statement, namely convergence of $Q^{\gamma,h}_{\hom}$
  and $\vartheta^h_i$ as $h\to 0$ with a linear rate in $h$. This requires the higher regularity
  $\|\vartheta_i^h\|_{H^2(\Box;\R^3)}\leq C$ for a constant $C$ independent of $h$,
  which in turn requires to assume that $\mathbb L(s,\cdot)$ is globally Lipschitz on $\Box$.
  Unfortunately, this assumption is often too restrictive, as it excludes typical composite materials
  with discontinuous elastic moduli.
  
  Our $\Gamma$-convergence proof for the fully discrete problem does not require the
  approximate microstructure problems to converge with a rate. However, such a rate
  would certainly be an important ingredient for any potential future proof of
  more than mere $\Gamma$-convergence.
\end{remark}

For the discretization of the macroscopic problem we require the coefficients $\mathbb L^{\gamma,h}_{\hom}(s)$ and $\Beffh(s)$ to be sufficiently regular in $\mac$, with bounds that are uniform in $h$. The main point of the following lemma is that $\vartheta^h_i$, $\Qhomh$ and $\Beffh$ inherit the local regularity in the variable~$s$ from  $\mathbb L$ and $B$.
The set $U$ will later play the role of a mesh element~$T$ of the macroscopic triangulation.
\begin{lemma}[Differentiability in $\mac$]
  Let $U\subseteq S$ be open and convex.
  Let $\mathbb L(\mac)$ satisfy the ellipticity bound \eqref{eq:Lbounds} for all $\mac\in U$, and assume that
  \begin{equation}\label{L:diff:refass}
    \mathbb L(\cdot,r), B(\cdot,r)\in W^{1,\infty}(U)\qquad\text{for all quadrature points $r\in R_h(\Box)$.}
  \end{equation}
  Then $\vartheta^h_i(\cdot)\in W^{1,\infty}(U;\mathcal V^h)$, $i=1,2,3$,
  $\mathbb L^{\gamma,h}_{\hom}(\cdot),B^{\gamma,h}_{\rm eff}(\cdot)\in W^{1,\infty}(S)$ and there exists a constant $C$ only depending on $c_0$, the reference quadrature scheme,
  and the limit $\gamma$ such that
  \begin{align*}
    \|\nabla_s\vartheta^h_i\|_{L^\infty(U)}+\|\nabla_s\mathbb L^{\gamma,h}_{\hom}\|_{L^\infty(U)}
    & \leq
    C\max_{r\in R_h(\Box)}\|\nabla_s\mathbb L(\cdot,r)\|_{L^\infty(U)},
    \\
    \\
    \|\nabla_s B^{\gamma,h}_{\rm eff}\|_{L^\infty(U)}
    & \leq C\max_{r\in R_h(\Box)}\Big\||\nabla_s\mathbb L(\cdot,r)||B(\cdot,r)|+|\nabla_sB(\cdot,r)|\Big\|_{L^\infty(U)},
    \\
  \end{align*}  
\end{lemma}

\begin{proof}
  In this proof we write $\langle \cdot,\,\cdot\rangle_{H^1(\Box)}$ for the $H^1$ scalar product
  in the finite-dimensional space~$\mathcal V^h$.
  Furthermore, we define $A(\mac)\in\operatorname{Lin}(\mathcal V^h;\mathcal V^h)$ and $F(\mac)\in\mathcal V^h$ by the identities
  \begin{align*}
    \langle A(\mac)u^h,\psi^h\rangle_{H^1(\Box)}
    & \colonequals \cG^h_\Box\Big[\big\langle\mathbb L(s,y)\nabla_\gamma u^h,\nabla_\gamma \psi^h\big\rangle\Big],
    \\
    \langle F(\mac),\psi^h\rangle_{H^1(\Box)}
    & \colonequals -\cG^h_\Box\Big[\big\langle\mathbb L(s,y)\iota(y_3 G_i),\nabla_\gamma \psi^h\big\rangle\Big],
  \end{align*}
  for every $u^h, \psi^h \in \mathcal V^h$.
  Since the space $\mathcal V^h$ has finite dimension, we conclude that \eqref{L:diff:refass}
  implies $A,F\in W^{1,\infty}(U)$, and for $j=1,2$ we have
  \begin{equation}\label{eq:idAderiv}
    \begin{aligned}
      \langle \partial_{s_j}A(\mac)u^h,\psi^h\rangle_{H^1(\Box)}
      & =
      \cG^h_\Box\Big[\big\langle \partial_{s_j}\mathbb L(s,y)\nabla_\gamma u^h,\nabla_\gamma \psi^h\big\rangle\Big],
      \\
      \langle \partial_{s_j} F(\mac),\psi^h\rangle_{H^1(\Box)}
      & =
      -\cG^h_\Box\Big[\big\langle \partial_{s_j}\mathbb L(s,y)\iota(y_3 G_i),\nabla_\gamma \psi^h\big\rangle\Big],
    \end{aligned}
  \end{equation}
  for a.e.~$\mac$ in $U$. Furthermore, by Lemma~\ref{L:approx:corr}, $A$ is uniformly invertible, i.e., for all $u^h\in\mathcal V^h$ we have $\|A^{-1}(\mac)u^h\|_{H^1(\Box)}\leq C\|u^h\|_{H^1(\Box)}$ where $C$ only depends on $c_0$ and on the local quadrature rule. Hence, since $A(\cdot)\in W^{1,\infty}(U)$, we conclude that $A^{-1}(\cdot)\in W^{1,\infty}(U)$ as well.
  
  To show the assertions, we first discuss the differentiability of the correctors $\vartheta^h_i(s)$ with respect to $s$.
  Note that their definition \eqref{eq:corrector_discrete} as the solution
  of a PDE takes the form $A(s)\vartheta^h_i(s)=F(s)$, and thus
  we have $\vartheta^h_i(s)=A^{-1}(\mac)F(\mac)$. The product rule for weak derivatives implies
  that $A^{-1}F\in W^{1,\infty}(U;\mathcal V^h)$, and thus the claimed regularity of $\vartheta^h_i$
  follows. To deduce the claimed estimates,
  we apply $\partial_{s_j}$ 
  to the equation $A(s)\vartheta_i^h=F$,
  and get
  \begin{align*}
    A\partial_{s_j}\vartheta^h_i & = \partial_{s_j}F-\partial_{s_j}A\vartheta^h.
  \end{align*}
  Now the claimed estimate for $\nabla_s\vartheta^h_i$ 
  follows from the bound
  on the operator norm of $A^{-1}$,
  the identities \eqref{eq:idAderiv}, and the definition of the quadrature scheme.
  The regularity and estimates for $\mathbb L^{\gamma,h}_{\hom}$ and $B^{\gamma,h}_{\rm eff}$ follow from the estimates on $\vartheta^h_i$ and the identities \eqref{eq:idAderiv} by a direct computation that again invokes the product rule for weak partial derivatives.
\end{proof}

\section{Discretization of the macroscopic problem}
\label{sec:MacroDiscretization}

In this section we present a discretization of the reformulated energy functional $\tenergy$
defined in~\eqref{eq:reformulated_energy},
which combines a discretization of the microscopic problem on a scale $h$
with a discretization of the macroscopic
bending problem on a mesh of size $H$. The discretization of $\tenergy$ will thus
be given by a discrete functional $\widetilde{\mathcal E}^{\gamma,h}_H$.
This new functional depends on $h$ via the approximate effective quantities,
but also on~$H$, because the discretization operates with approximate
differential operators and uses numerical quadrature.
We show that $\widetilde{\mathcal E}^{\gamma,h}_H$ $\Gamma$-converges to $\tenergy$
whenever $(h,H)\to 0$.

For the macroscopic discretization the main difficulty is the treatment of the
isometry constraint~\eqref{IsometryConstraint}. We follow Bartels \cite{bartels2013approximation,bartels2017bilayer} and use a
Discrete Kirchhoff Triangle (DKT) finite element discretization.
Such a discretization uses $H^1$-deformations that satisfy the isometry constraint exactly
at the mesh vertices, but not elsewhere.
The resulting discretization is therefore non-conforming in the sense that the
discrete deformations are neither elements of $H^2(S;\R^3)$
nor of $H^1_\iso(S;\R^3)$.

\subsection{Discrete Kirchhoff Triangle finite elements}\label{sec:discreteDKT}

We give a brief summary of discrete Kirchhoff triangles and refer to~\cite{batoz1980study,braessfem,bartels2013approximation} for more details.
Denote by $\macrogrid$ a shape-regular family of triangulations of $S$ parametrized by
the maximum edge length $H$.
We denote the set of vertices of $\macrogrid$ by $\cN_H$,
and by $H_T$ the diameter of any $T\in \macrogrid$.

The construction of discrete Kirchhoff triangles is based on two $H^1$-conforming finite element spaces 
that are coupled via an approximate first derivative $\nabla_H$ called
discrete gradient operator.
While the spaces are typically given for scalar-valued functions, we give here
a direct definition for vector fields.

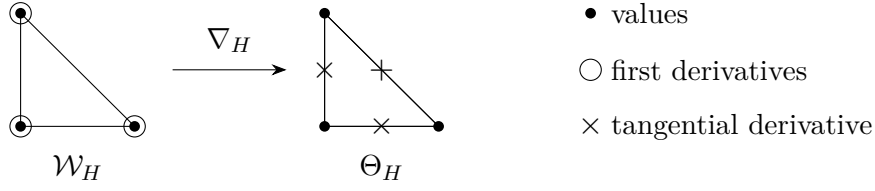
\begin{figure}
  \begin{center}
    \begin{tikzpicture}
      \coordinate (L1) at (0, 0);
      \coordinate (L2) at (1.5, 0);
      \coordinate (L3) at (0, 1.5);
      
      \draw (L1) -- (L2) -- (L3) -- cycle;
      
      \draw (L1) circle (4pt);
      \fill (L1) circle (2pt);
      \draw (L2) circle (4pt);
      \fill (L2) circle (2pt);
      \draw (L3) circle (4pt);
      \fill (L3) circle (2pt);
      
      \node[below=of L1.south, xshift=0.75cm, yshift=0.75cm] {$\cW_H$};
      
      \coordinate (R1) at (4, 0);
      \coordinate (R2) at (5.5, 0);
      \coordinate (R3) at (4, 1.5);
      
      \draw (R1) -- (R2) -- (R3) -- cycle;
      
      \fill (R1) circle (2pt);
      \fill (R2) circle (2pt);
      \fill (R3) circle (2pt);
      
      \draw (R1) -- (R2) node[midway] {\large$\times$};
      \draw (R2) -- (R3) node[midway, sloped] {\large$\times$};
      \draw (R3) -- (R1) node[midway] {\large$\times$};
      
      \node[below=of R1.south, xshift=0.75cm, yshift=0.75cm] {$\Theta_H$};
      
      \draw[-{Stealth[length=2mm]}] (2.0, 0.75) -- (3.5, 0.75);
      
      \node[above=0.1cm, xshift=2.75cm, yshift=0.75cm] {$\nabla_H$};
      
      \coordinate (Le1) at (7.5,0);
      \coordinate (Le2) at (7.5,0.75);
      \coordinate (Le3) at (7.5,1.5);
      
      \node at (Le1) {\large $\times$};
      \node [label=right:{tangential derivative}] at (Le1) {};
      \draw (Le2) circle (4pt);
      \node [label=right:{first derivatives}] at (Le2) {};
      \fill (Le3) circle (2pt);
      \node [label=right:{values}] at (Le3) {};
    \end{tikzpicture}
  \end{center}
  \caption{Degrees of freedom of the deformation space $\mathcal{W}_H$ and
  the image of $\nabla_H$ in the rotation space $\Theta_H$}
  \label{fig:DKT-spaces}
\end{figure}

Let $\cP_k(T;\R^3)$ be the space of $\R^3$-valued $k$th-order polynomials on the triangle~$T$.
Deformations are approximated in the space
\begin{multline*}
  \cW_H \colonequals \Big\{w_H \in C(\bar{S};\R^3)\; : \;  w_{H|T} \in \pRed(T;\R^3) \text{ for all $T\in \macrogrid$}, \\
  \text{and $\nabla w_H$ is continuous at the $\cN_H$}\Big\},
\end{multline*}
where $\pRed(T;\R^3)$ denotes the space of $\R^3$-valued reduced cubic polynomials defined on $T\in\macrogrid$ by
\begin{equation*}
  \pRed(T;\R^3)
  \colonequals
  \Big\{p \in \cP_3(T;\R^3) \; : \;
  6p(z_T) = \sum_{i=1}^{3} \big(2p(z_i)- \nabla p(z_i) \cdot (z_i - z_T)\big)  \Big\},
\end{equation*}
with $z_T$ the center of mass of $T$.

The local space $\pRed(T;\R^3)$ is 27-dimensional, and the function values
and partial derivatives at the nodes of the triangulation form a unisolvent set of degrees of freedom
for the space $\cW_H$ (see Figure~\ref{fig:DKT-spaces}).
Using these degrees of freedom, we can define a nodal interpolation operator
$\cI_H^{\cW}: H^3(S;\R^3) \to \cW_H$ on $T\in \macrogrid$ via the conditions
$\cI_H^\cW[w](z) = w(z)$ and $\nabla\cI_H^\cW[w](z) = \nabla w(z)$ for all $z \in \cN_H\ \cap T$.\footnote{Note that the interpolation $\cI_H^\cW$ is well-defined on $H^3(S)$ due to
the embedding $H^3(S) \hookrightarrow C^1(S)$.
} 
By \cite[Theorem~4.4.4]{brennerScott2008mathematical}
we have for all $w\in H^{3}(T;\R^3)$ and $T\in \macrogrid$ the interpolation estimate
\begin{equation}
  \label{eq:DKTinterpolationEstimate}
  \norm{w - \cI_H^{\cW}w}_{L^3(T)} + H_T  \norm{\nabla w - \nabla \cI_H^{\cW}w}_{L^3(T)} + H_T^2  \norm{\nabla^2w - \nabla^2\cI_H^{\cW}w}_{L^3(T)}
  \leq C H_T^3 \norm{\nabla^3w}_{L^3(T)}
\end{equation}
with a constant $C > 0$ independent of  $H_T>0$.

The main point of DKT elements is that derivatives of discrete deformations are
replaced by a particular approximation. For this, the DKT element uses a second space
\begin{align*}
  \Theta_H & \colonequals \left\{ \theta_H \in C(S; \R^{3 \times 2}) \; : \; \theta_{H}|_T \in \cP_2(T;\R^{3 \times 2})
  \text{ for all $T\in \macrogrid$} \right\},
\end{align*}
and defines a custom gradient operator $\nabla_H : \mathcal{W}_H \to \Theta_H$.
Let $\mathcal{E}_H$ be the set of edges of the triangulation, and for each edge $e \in \mathcal{E}_H$
call $z_e^1, z_e^2$ the endpoints, and  $t_e,n_e$, and $z_e$ the unit tangent,
unit normal vector and edge midpoint of $e$, respectively.

\begin{definition}[Discrete gradient operator $\nabla_H$]
  Let $w_H \in \cW_H$. Then the discrete gradient $\theta_H \colonequals \nabla_H w_H \in \Theta_H$ of $w_H$ is uniquely defined by the
  following conditions: 
  \begin{align*}
    \theta_H(z) &= \nabla w_H(z)  &\forall z &\in \cN_H \\
    \theta_H(z_e) \cdot t_e &= \nabla w_H(z_e) \cdot t_e  &\forall e &\in \cE_H\\ 
    \theta_H(z_e)\cdot n_e &= \frac{1}{2} \left(\nabla w_H(z_e^1) + \nabla w_H(z_e^2) \right) \cdot n_e
    &\forall e &\in \cE_H.
  \end{align*}
\end{definition}

The operator $\nabla_H : \cW_H \to \Theta_H$ is linear.
It is, however, not surjective: Its range is a subspace of $\Theta_H$
consisting of functions with an affine normal derivative along the edges.

The discrete gradient operator $\nabla_H$ admits several desirable properties, cf.~\cite{braessfem,bartels2015numerical,bartels2022stable}:
In particular,
\begin{align}\nonumber
  \frac{1}{C} \norm{\nabla w_H}_{L^2(\dom)} 
  &\leq \norm{\nabla_H w_H}_{L^2(\dom)}   \leq C \norm{\nabla w_H}_{L^2(\dom)},\\
  \shortintertext{and}
  \norm{\nabla_H w_H - \nabla w_H}_{L^2(T)}
  &\leq C H_T \norm{\dhess w_H}_{L^2(T)}\label{eq:discGradProp4}
\end{align}
for all $w_H\in \cW_H$ and $T\in\macrogrid$. The constant $C>0$ is independent of
the mesh and element sizes $H$ and $H_T$, respectively.

Note that the constants in the local estimates \eqref{eq:DKTinterpolationEstimate} and \eqref{eq:discGradProp4} depend on the shape-regularity of the element $T\in\macrogrid$. However, since we assume that the family of triangulations $\macrogrid$ is shape-regular, these constants can be chosen uniformly for all $T\in\macrogrid$ and all $H>0$.

\subsection{The fully discrete energy functional}

We now apply the DKT discretization to the reformulated energy functional $\tenergy$
of~\eqref{eq:reformulated_energy}, and obtain the discrete energy functional
\begin{equation}
  \label{eq:discreteEnergy}
  \dtenergy(\deform_H)
  \colonequals
  \begin{cases}
    \macroquad{\Big\langle\boldsymbol{L}^{\gamma,h}_{\hom}(\mac)\dhess \deform_H,\,\big(\dhess \deform_H+2n_H\otimes B^{\gamma,h}_{\rm eff}(\mac)\big)\Big\rangle} & \text{if $\deform \in \cA_{H}$}, \\
    +\infty & \text{else}.
  \end{cases}
\end{equation}
The symbol $\cG^H_S$ denotes a quadrature operator approximating integration over $S$.
Note how the discrete quantity $\dhess v_H$ replaces the Hessian $\nabla^2 v_H$.
Likewise, the actual surface normal is replaced by
\begin{equation} \label{eq:discreteNormalField}
  n_H \colonequals \lagint\big[\partial_1^H v_H \times \partial_2^H v_H\big],
\end{equation}
where $\partial_1^H v_H$ and $\partial_2^H v_H$ denote the first and second column
of $\dgrad v_H$, respectively.
The operator~$\lagint$ denotes standard first-order Lagrange interpolation on $\macrogrid$.
It is applied to ensure that $\abs{n_H} \le 1$ everywhere---equality will typically hold
only at the nodes of $\mathcal{T}_H$.
The admissible set $\cA_H$ consists of all discrete deformations that satisfy
the isometry constraint and the boundary conditions at the nodes of the triangulation $\macrogrid$
\begin{multline*}
  \cA_H \colonequals
  \Big \{
  v_H \in \cW_H \; : \;
  v_{H}(z) = \widetilde u_{D}(z),\
  \nabla v_H(z) = \nabla\widetilde u_D(z) \text{ for all } z\in \cN_H\cap \Gamma_D,
  \\
  \nabla v_H(z)^T \nabla v_H(z) =  I_{2\times 2} \ \text{ for all } z\in \cN_H
  \Big \}.
\end{multline*}

It remains to discuss the quadrature operator $\cG^H_S$ used
in~\eqref{eq:discreteEnergy} to approximate the integral $\int_S\dd\mac$.
To construct such an operator, for every triangle $T \in \macrogrid$ we fix a quadrature rule
$\{(\mu^T_i,q^T_i)\}_{i=1,\ldots,n}$ consisting of quadrature points $q^T_i$ and weights $\mu^T_i>0$. 
We then define the local and global quadrature operators by
\begin{alignat*}{2}
  \cG_T^H & : C^0(T;\R) \to \R,          & \qquad \cG_T^H[f] &\colonequals  \sum_{i=1}^{n} \mu_i^T f(r^T_i),
  \\
  \cG_\dom^H & : C^0(\dom;\R) \to \R,    &\cG_\dom^H[f] &\colonequals \sum_{T\in \macrogrid}^{} \cG_T^H[f].
\end{alignat*}

\begin{assumption}[Exactness of $\cG_T^H$ and $\cG_\dom^H$]
  \label{assumption:macroQuadrature}
  The element quadrature operator $\cG_T^H$ is exact on quadratic polynomials for all $T\in\macrogrid$.
\end{assumption}

\begin{remark}[Existence of minimizers of $\dtenergy$]
  The functional $\dtenergy$ is continuous on $\cA_H$, which is a non-empty and 
  closed subset of the finite-dimensional normed vector space $\cW_{H}$.
  From Theorem~\ref{T:main}\ref{item:compactness} below we further deduce that $\dtenergy$ is coercive.
  Hence, by the Weierstrass theorem, $\dtenergy$ attains its minimum on $\cA_H$.
\end{remark}

\subsection{\texorpdfstring{$\Gamma$}{Gamma}-convergence for \texorpdfstring{$(h,H)\to 0$}{(h,H) to 0}}

We now prove $\Gamma$-convergence of the fully discrete energy functionals $\dtenergy$
as $(h,H)\to 0$. 
Importantly, our argument does not require a convergence rate of the microstructure
discretization with respect to~$h$.
It turns out that with regard to the discretization on scale $h$, we only require the approximate bending stiffness $\mathbb L^{\gamma,h}_{\hom}$ to be uniformly (in $h$) elliptic,
the prestrain $B^{\gamma,h}_{\rm eff}$ to be uniformly bounded,
and both to converge pointwise to $\mathbb L^\gamma_{\hom}$
and $B^\gamma_{\rm eff}$, respectively---our argument does not rely on the details of the
approximation in $h$.
To highlight this in the theorem below, we formulate it in terms of general,
converging sequences $\mathbb L^h$ and $B^h$, which generalize the quantities
$\mathbb L^{\gamma,h}_{\hom}$ and $B^{\gamma,h}_{\rm eff}$ of Chapter~\ref{sec:MicroNumerics}.

In the following we use the DKT finite element discretization for the
macroscopic problem as introduced in the previous chapters, and
Assumptions~\ref{assumption:macroQuadrature} on the numerical quadrature.
Furthermore, for the handling of Dirichlet boundary conditions we suppose:
\begin{assumption}[Boundary conditions]
  \label{Assumption:BoundaryData}  \mbox{} 
  \begin{enumerate}[label=(\alph*)]
    \item $\dom\subseteq\R^2$ is a bounded polygonal domain.
    \item \label{item:bd_approximability}
    \emph{(Approximability)}.
    The set $\cA\cap H^3(S;\R^3)$ is dense in $\cA$ in the strong topology of $H^2(S;\R^3)$.
    \item \emph{(Compatibility)}.
    $\cI_H^{\cW}[\cA\cap H^3(S;\R^3)]\subset\cA_H$.
  \end{enumerate}
\end{assumption}

Recall that $\macrogrid$ is a shape-regular family of triangulations of $S$,
parametrized by the maximum edge length $H$.

\begin{theorem}[$\Gamma$-convergence]\label{T:main}
  For every $h\geq 0$ let $\mathbb L^h:S\to\operatorname{Lin}(\R^{2\times 2},\R^{2\times 2})$
  and $B^h:S\to\R^{2\times 2}_{\sym}$ be measurable, and such that
  $\langle\mathbb L^h(\mac)F,G\rangle=\langle\mathbb L^h(\mac)\sym F,\sym G\rangle$
  for all $\mac\in S$ and for all $F,G\in\R^{2\times 2}$.
  Further assume:
  \begin{enumerate}[(H1)]
    \item \label{ass:uniform_ellipticity_and_boundedness}
    (Uniform ellipticity and boundedness). For all $\mac\in S$ and for all $F,G\in\R^{2\times 2}$,
    \begin{align*}
      \langle\mathbb L^h(\mac)F,G\rangle
      & \leq
      c_S|\sym F||\sym G|,
      \\
      \langle\mathbb L^hG,G\rangle
      & \geq
      \frac1{c_S}|\sym G|^2,
      \\
      |B^h(\mac)|
      & \leq
      c_S,
    \end{align*}
    with a constant $c_S>0$ independent of $h$ and $\mac$. 
    \item (Local $W^{1,\infty}$-regularity on the meshes).
    \begin{equation*}
      \lim_{H\to 0}H\max_{T\in\mathcal T_{H}}\|\mathbb L^{h}\|_{W^{1,\infty}(T)}(1+\|B^{h}\|_{W^{1,\infty}(T)})\to 0\text{ uniformly in $h$},
    \end{equation*}
    where we assume that the Sobolev norms of $\mathbb{L}^h$ and $B^h$ are well-defined
    on all mesh elements of the family $\macrogrid$.
    
    \item (Convergence). $\mathbb L^h\to\mathbb L^0$ pointwise on $S$ and $B^h\to B^0$ in $L^2(S)$.
  \end{enumerate}
  Introduce the third-order stiffness $\boldsymbol{L}^h(\mac)\in\operatorname{Lin}(\R^{3\times 2\times 2},\R^{3\times 2\times 2})$ by
  \begin{equation*}
    \langle\boldsymbol{L}^h(\mac) \mathbf{F} , \mathbf{G}\rangle
    \colonequals
    \sum_{i=1}^3 \Big \langle\mathbb L^h(\mac) (e_i\cdot \mathbf{G}),\,e_i\cdot \mathbf{F}\Big\rangle
    \quad\text{for all $\mathbf F, \mathbf G\in\R^{3\times 2\times 2}$},
  \end{equation*}
  and define the energy functionals $\mathcal E^h_H,\mathcal E^0:L^2(S;\R^3)\to\R\cup\{+\infty\}$ by
  \begin{align*}
    \mathcal E^h_H(\deform_H)
    & \colonequals \cG^{H}_S\Big[\Big\langle \boldsymbol{L}^{h}(\mac)\nabla\nabla_{H}\deform_H,\,\nabla\nabla_{H}\deform_H+2n_H \otimes B^{h}(\mac)\Big\rangle\Big],
    \\
    \mathcal E^0(\deform)
    & \colonequals
    \int_S \Big\langle\boldsymbol{L}^0(\mac)\nabla^2\deform, \nabla^2\deform+2 n_{\deform}\otimes B^0(\mac)\Big\rangle\dd s,
  \end{align*}
  for $\deform_H\in\cA_{H}$ and $\deform\in\cA$, and $+\infty$ otherwise.
  Then for any sequence $(h_j,H_j)\to 0$ as $j\to\infty$ the following holds:
  \begin{enumerate}[label=(\alph*)]
    \item \label{item:compactness}
    \emph{(Compactness)}. Let $(\deform_j)$ be a sequence with $\deform_j\in\cA_{H_j}$ and equibounded energy, i.e., 
    \begin{equation}\label{eq:EquicoercivityBound}
      \limsup_{j\to\infty}\mathcal E^{h_j}_{H_j}(\deform_j)<\infty.
    \end{equation}
    Then
    \begin{equation*}
      \limsup\limits_{j\to\infty}\norm{\nabla\nabla_{H_j} \deform_j}_{L^2(\dom;\R^{3\times2\times2})}<\infty.
    \end{equation*}
    In particular, there exists a subsequence (not relabeled) and $\deform\in\cA$ such that
    \begin{equation}\label{eq:comp:conv1}
      \deform_j\to\deform\text{ strongly in }H^1(S),\qquad      \nabla_{H_j}\deform_j\wto \nabla \deform\text{ weakly in $H^1(S)$},
    \end{equation}
    and
    \begin{equation}\label{eq:comp:conv2}
      -n_{H_j}\cdot\nabla\nabla_{H_j}\deform_j\wto \II_{\deform}\qquad\text{weakly in }L^2(S).
    \end{equation}
    \item \label{item:lower_bound}
    \emph{(Lower bound)}. For all sequences $\deform_j\to \deform$ in $L^2(S;\R^3)$, we have
    \begin{equation*}
      \liminf_{j\to \infty} \mathcal E^{h_j}_{H_j}(\deform_j) \geq \mathcal E^0(\deform).
    \end{equation*}
    \item \label{item:recovery_sequence}
    \emph{(Recovery sequence)}. For all $\deform\in \cA$ there exists a sequence $\deform_j\in \cA_{H_j}$ with
    $\deform_j\to \deform$ strongly in $H^1(\dom;\R^3)$ such that
    \begin{equation*}
      \lim_{j\to\infty}\mathcal E^{h_j}_{H_j}(\deform_j)= \mathcal E^0(\deform).
    \end{equation*}
  \end{enumerate}
\end{theorem}

In the proof we first relate the fully discrete energies $\mathcal E^{h_j}_{H_j}$
to semi-discrete energies where the numerical quadrature $\cG^H_{S}$ is replaced
by actual integration. The associated quadrature error is estimated with the help of the following technical lemma, whose proof we postpone to the end of the section.

\begin{lemma}[Quadrature estimate]
  \label{lem:overall_convergence_quad_error}
  For integrable $f:S\to\R$ set $E_S^H[f]\colonequals\cG^H_S[f]-\int_Sf\,\dd s$.
  Under the assumptions of Theorem~\ref{T:main} for every $\deform_H\in\cA_{H}$ we have
  \begin{multline*}
    \left|E_S^H\Big[
    \Big\langle\boldsymbol L^{h}(\mac)\nabla\nabla_{H}\deform_H, \nabla\nabla_{H}\deform_H+2n_{H}\otimes B^{h}(\mac)\Big\rangle\Big]\right|
    \\
    \leq C H\max_{T\in\mathcal T_{H}}\Big(\|\mathbb L^{h}\|_{W^{1,\infty}(T)}\big(1+\|B^{h}\|_{W^{1,\infty}(T)}\big)\Big)\Big(\|\nabla\nabla_{H}\deform_H\|^2_{L^2(S)}+|S|^\frac12\|\nabla\nabla_{H}\deform_H\|_{L^2(S)}\Big),
  \end{multline*}
  where the constant $C$ is independent of $h$ and $H$.
\end{lemma}

We then treat the semi-discrete energies using the following
abstract (lower semi-)continuity result for parametrized quadratic functionals
on finite-dimensional vector spaces.

\begin{lemma}\label{L:lowerbound_abstract}
  Let $\widetilde{\mathbb L}_{j},\widetilde{\mathbb L}\in L^\infty(S;\operatorname{Lin}(\R^N;\R^N))$
  and $F_{j},{F}\in L^2(S;\R^N)$ for all $j \in \mathbb N$, and consider the functionals
  \begin{alignat*}{2}
    \mathcal F_{j} & :L^2(S;\R^N)\to\R, & \qquad\mathcal F_{j}(G)& \colonequals\int_S\langle\widetilde{\mathbb L}_{j}G,G+{F}_{j}\rangle\dd s,
    \\
    \mathcal F & :L^2(S;\R^N)\to\R, & \mathcal F(G) & \colonequals\int_S\langle\widetilde{\mathbb L}G,G+{F}\rangle\dd s.
  \end{alignat*}
  Assume that
  \begin{equation*}
    \limsup_{j\to \infty}\big(\|\widetilde{\mathbb L}_{j}\|_{L^\infty(S)}+\|{F}_{j}\|_{L^2(S)}\big)<\infty
    \qquad \text{and} \qquad
    \langle\widetilde{\mathbb L}_{j}G,G\rangle\geq 0,
  \end{equation*}
  and
  \begin{equation*}
    \widetilde{\mathbb L}_{j} \to\widetilde{\mathbb L}\qquad\text{a.e.\ in $S$},
    \qquad \qquad
    {F}_{j} \to{F}\qquad\text{in }L^2(S).
  \end{equation*}
  Then $G_j \wto G$ weakly in $L^2(S;\R^N)$ implies $\liminf\limits_{j \to \infty}\mathcal F_{j}(G_{j})\geq\mathcal F(G)$.
  If even $G_{j}\to G$ strongly in $L^2(S;\R^N)$, then $\lim\limits_{j \to \infty}\mathcal F_{j}(G_{j})=\mathcal F(G)$.
\end{lemma}
The proof of this lemma will also appear at the end of this section.
We proceed with the proof of Theorem~\ref{T:main}.

\begin{proof}[Proof of Theorem~\ref{T:main}] \mbox{} 
  \begin{enumerate}[wide,label=\textbf{Step~\arabic*}, labelindent=0pt]
    \item (Compactness)
    In this first step, we prove an energy estimate and then the
    compactness part~\ref{item:compactness} of the theorem.
    \begin{enumerate}[wide,label=\theenumi\textbf{.\arabic*}, labelindent=0pt]
      \item \label{proofstep:macroscopic_energy_estimate} (Energy estimate). 
      We first claim that for all $u_H\in\cA_H$ we have a constant $\widetilde C > 0$
      depending only on the constant $c_S$ of Assumption~\ref{ass:uniform_ellipticity_and_boundedness} (and in particular not on $H$) such that
      \begin{equation}
        \label{eq:macroscopic_energy_estimate}
        \mathcal E^{h}_H(u_H)
        \geq \frac{1}{\widetilde C}\|\sym\dhess u_H\|_{L^2(S)}^2-\widetilde C,
      \end{equation}
      where $\sym\dhess u_H$ denotes the third-order tensor with coefficients $(\sym\dhess u_H)_{ijk}\colonequals\frac12(\partial_j\partial_{H,k}u_{H,i}+\partial_k\partial_{H,j}u_{H,i})$.
      To show~\eqref{eq:macroscopic_energy_estimate}, recall the definition of
      \begin{equation*}
        \mathcal E^h_H(\deform_H)
        \colonequals
        \cG^{H}_S\Big[\Big\langle \boldsymbol{L}^{h}(\mac)\nabla\nabla_{H}\deform_H,\,\nabla\nabla_{H}\deform_H+2n_H \otimes B^{h}(\mac)\Big\rangle\Big],
      \end{equation*}
      and note that it can be split into two parts. To bound the first part,
      by the ellipticity of $\mathbb L^h$ and the exactness of the quadrature scheme
      on second-order polynomials (Assumption~\ref{assumption:macroQuadrature})
      we obtain the lower bound
      \begin{align*}
        \macroquad{\Big\langle\boldsymbol{L}^h(\mac)\dhess u_H,\,\dhess u_H\Big\rangle}
        &\geq
        \frac{1}{c_s}\macroquad{\Big\langle\sym\dhess u_H,\,\sym\dhess u_H\Big\rangle}
        \\
        & =
        \frac{1}{c_s}\norm{\sym\dhess u_H}_{L^2(\dom)}^2.
      \end{align*}
      For the second part, using the bounds on $\mathbb L^h$, $B^h$, and the pointwise bound $|n_h|\leq 1$ (which holds by construction~\eqref{eq:discreteNormalField}), we conclude that
      \begin{multline*}
        \left|\macroquad{\langle\boldsymbol{L}^h(\mac)\dhess u_H,\,n_H\otimes B^h\rangle}\right|
        \leq c_S\left(\macroquad{|\sym\dhess u_H|^2}\right)^\frac12
        \left(\macroquad{|n_H\otimes B^h|^2}\right)^\frac12\\
        \leq c_S\left(\macroquad{|\sym\dhess u_H|^2}\right)^\frac12
        \|n_H\|_{L^\infty(S)}\|B^h\|_{L^\infty(S)}
        \leq c_S^2\|\sym\dhess u_H\|_{L^2(S)},
      \end{multline*}
      where we have used the exactness of the quadrature scheme for the first term, and the bounds $\|n_H\|_{L^\infty(S)}\leq 1$ and $\|B^h\|_{L^\infty(S)}\leq c_S$ (cf.~\ref{ass:uniform_ellipticity_and_boundedness}).
      The claim~\eqref{eq:macroscopic_energy_estimate} then follows by combining both estimates.
      
      \item (Compactness).
      We then argue that the equiboundedness of the energy $\mathcal E^{h_j}_{H_j}(\deform_j)$, i.e.,~\eqref{eq:EquicoercivityBound},
      implies
      that the sequence $(\deform_j)$ is bounded in $H^2(S)$. Indeed, by \ref{proofstep:macroscopic_energy_estimate}
      we have $$\limsup_{j\to\infty}\|\sym\nabla\nabla_{H_j}\deform_j\|_{L^2(S)}^2<\infty.$$
      Using that $\nabla_{H_j}\deform_j$ has prescribed boundary values on the nodes $\mathcal N_H\cap\Gamma_D$, we conclude
      with the help of Korn's inequality that
      \begin{equation}\label{eq:equib1}
        \limsup_{j\to\infty}\big(\|\nabla_{H_j}\deform_j\|_{L^2(S)}+\|\nabla\nabla_{H_j}\deform_j\|_{L^2(S)}\big)<\infty. 
      \end{equation}
      Furthermore, the approximation property~\eqref{eq:discGradProp4} of the
      discrete gradient $\nabla_{H_j}$ implies
      \begin{equation}\label{eq:discr}
        \|\nabla_{H_j}\deform_j-\nabla\deform_j\|_{L^2(S)}\leq CH_j\|\nabla\nabla_{H_j}\deform_j\|_{L^2(S)}.
      \end{equation}
      Together with \eqref{eq:equib1} and the triangle inequality we deduce that
      the sequence $(\nabla\deform_j)$ is bounded in $L^2(S)$. Since the values
      of the $\deform_j$ are prescribed on the nodes $\mathcal N_H\cap\Gamma_D$ as well,
      Poincar\'e's inequality applies, and we conclude that
      \begin{equation}\label{eq:equib2}
        \limsup_{j\to\infty}\|\deform_j\|_{H^1(S)}<\infty. 
      \end{equation}
      Note that~\eqref{eq:equib1} and \eqref{eq:equib2} allow us to extract a subsequence (not relabeled) such that
      \begin{align*}
        \deform_j\wto\deform\text{ and }
        \nabla_{H_j}\deform_j\wto \theta \text{ weakly in } H^1(S),
      \end{align*}
      with $\deform\in H^1(S;\R^3)$ and $\theta \in H^1(\dom;\R^{3\times2})$ satisfying $\deform=\widetilde u_D$ and $\theta = \nabla\widetilde u_D$ on $\Gamma_D$
      (thanks to the continuity of the trace operator.)
      Furthermore, \eqref{eq:discr} combined with \eqref{eq:equib1} implies that $\theta=\nabla\deform$.
      Since the embedding $H^1(S)\hookrightarrow L^2(S)$ is compact, we find that $\nabla_{H_j}\deform_j\to\nabla\deform$ strongly in $L^2(S)$.
      Now, \eqref{eq:discr} implies that $v_j\to v$ strongly in $H^1(S)$, which proves \eqref{eq:comp:conv1}.
      
      Following the argument of \cite{bartels2022stable}
      we then deduce that $\deform\in H^2_{\rm iso}(S;\R^3)$ and thus $\deform\in\cA$.
      To see this, let $\cI_{H_j}^1$ denote the standard first-order Lagrange interpolation
      operator on $\mathcal T_{H_j}$. Then, using that $\deform_j\in\cA_{H_j}$ is a
      nodewise isometry and an interpolation estimate, we have
      \begin{align*}
        \norm[\big]{(\nabla_{H_j} \deform_j)^\top\nabla_{H_j} \deform_j - I_{2\times2}}_{L^1(\dom)}
        &\leq
        \norm[\big]{(\nabla_{H_j} \deform_j)^\top\nabla_{H_j} \deform_j - I_{H_j}^1[(\nabla_{H_j} \deform_j)^\top\nabla_{H_j}\deform_j]}_{L^1(\dom)} \\
        &\quad +  \norm[\big]{I_{H_j}^1[(\nabla_{H_j} \deform_j)^\top\nabla_{H_j} \deform_j] - I_{2\times2}}_{L^1(\dom)} \\
        &\leq
        C H \norm[\big]{\nabla\big((\nabla_{H_j}\deform_j)^\top\nabla_{H_j} \deform_j \big)}_{L^1(\dom)}\\
        &\leq
        C H\norm[\big]{\nabla \nabla_{H_j} \deform_j}_{L^2(\dom)} \norm{\nabla_{H_j} \deform_j}_{L^2(\dom)},
      \end{align*}
      where the last estimate follows from the product rule and the Cauchy--Schwarz inequality.
      Thanks to \eqref{eq:equib1} the right-hand side converges to zero as $j\to\infty$. Since $(\nabla_{H_j} \deform_j)^\top\nabla_{H_j} \deform_j\to \nabla\deform^\top\nabla\deform$ strongly in $L^1(\dom;\R^{2\times 2})$, we deduce that $\nabla\deform^\top\nabla\deform=I_{2\times 2}$ a.e.\ in $S$.
      
      It remains to prove \eqref{eq:comp:conv2}. 
      Since $\nabla\nabla_{H_j}\deform_j$ is a bounded sequence in $L^2(S)$ and because
      $|n_{H_j}|\leq 1$ in~$S$, the sequence $n_{H_j}\cdot\nabla\nabla_{H_j}\deform_j$ is bounded in $L^2(S)$ as well. It thus converges weakly in $L^2(S)$ to some limit $A\in L^2(S)$ up to a subsequence. Since $n_{H_j}\to n_{\deform}$ a.e.~in $S$ (up to a subsequence) and $\nabla\nabla_{H_j}\deform_j\wto\nabla^2\deform$ weakly in $L^2(S)$, we conclude that $A=n_{\deform}\cdot \nabla^2\deform=-\II_{\deform}$. Since this uniquely determines $A$, the entire sequence converges weakly in $L^2(S)$.
      
      At this point we have shown assertion~\ref{item:compactness}.
    \end{enumerate}
    
    \item (Lower bound)
    
    We next prove part~\ref{item:lower_bound} of Theorem~\ref{T:main}, i.e.,
    the assertion that for all sequences $\deform_j\to \deform$ in $L^2(S;\R^3)$ we have
    \begin{equation*}
      \liminf_{j\to \infty} \mathcal E^{h_j}_{H_j}(\deform_j) \geq \mathcal E^0(\deform).
    \end{equation*}
    For this, it suffices to consider the case where the left-hand side is bounded, since otherwise the inequality is trivial.
    As the energy is bounded from below by \ref{proofstep:macroscopic_energy_estimate},
    we may pass to a subsequence (not relabeled)
    such that $\lim_{j\to\infty}\mathcal E^{h_j}_{H_j}(\deform_j) = \liminf_{j\to\infty}\mathcal E^{h_j}_{H_j}(\deform_j)$.
    By using the quadrature error bounds of Lemma~\ref{lem:overall_convergence_quad_error} we deduce that $\deform\in\cA$,
    and that $(\deform_j)$ converges to $\deform$ in the sense of \eqref{eq:comp:conv1}
    and~\eqref{eq:comp:conv2}.
    In particular, we conclude that as $j\to \infty$
    \begin{alignat*}{2}
      \nabla\nabla_{H_j}\deform_j & \wto \nabla^2\deform & \qquad\text{weakly in $L^2(S)$},
      \\
      n_{H_j}\otimes B^{h_j} & \to n_{\deform}\otimes B^0 & \text{strongly in $L^2(S)$}.
    \end{alignat*}
    An application of Lemma~\ref{L:lowerbound_abstract} with $G_j\colonequals\nabla\nabla_{H_j}\deform_j$
    and $F_j\colonequals n_{H_j}\otimes B^{h_j}$ thus yields
    \begin{multline*}
      \liminf_{j\to \infty}\int_S \Big\langle\boldsymbol{L}^{h_j}(\mac)\nabla\nabla_{H_j}\deform_j, \nabla\nabla_{H_j}\deform_j+2n_{H_j}\otimes B^{h_j}(\mac)\Big\rangle \dd\mac
      \\
      \geq
      \int_S \Big\langle\boldsymbol{L}^0(\mac) \nabla^2\deform, \nabla^2\deform+2n_{\deform}\otimes B^0(\mac)\Big\rangle \dd\mac
      =
      \mathcal E^0(\deform).
    \end{multline*}
    Combined with Lemma~\ref{lem:overall_convergence_quad_error} again,
    the asserted lower bound follows.
    
    \item
    (Recovery sequence)
    
    In this final step we prove assertion~\ref{item:recovery_sequence} of Theorem~\ref{T:main},
    i.e., we construct a recovery sequence.
    In the first substep, we only do this for the case
    that $\deform\in\mathcal A\cap H^3(S)$. To that end we construct
    the sequence elements $\deform_{j}$ by a suitable interpolation and prove
    strong convergence properties that imply the convergence of the energy functional
    along the sequence. In the second substep we construct a recovery sequence
    for general $\deform\in\mathcal A$ by appealing to a diagonal-sequence argument.
    
    \begin{enumerate}[wide,label=\theenumi\textbf{.\arabic*}, labelindent=0pt]
      \item (Recovery sequence for $\deform\in\mathcal A\cap H^3(S)$).\label{proofstep:recov_reg}
      Let $\deform\in\mathcal A\cap H^3(S)$. We first claim that there exists a sequence $\deform_{j}\in\cA_{H_j}$ such that
      \begin{alignat*}{2}
        \deform_{j} & \to \deform \qquad & \text{strongly in }H^1(S)\\
        \nabla\nabla_{H_j}\deform_{j} & \to \nabla^2 \deform & \text{strongly in }L^2(S).
      \end{alignat*}
      The construction of such a $(\deform_j)$ follows the arguments in~\cite{bartels2013approximation,bartels2022stable,rumpf2023two,bartels2017bilayer},
      which we recall here for the readers' convenience.
      Thanks to the $H^3$-regularity of $\deform$ the interpolation operator $\cI_H^{\mathcal W}$
      is well-defined, and we can set
      $\deform_j \colonequals \cI_{H_j}^{\mathcal W} \deform \in \cW_H \subset H^1(\dom;\R^3)$.
      
      By definition of the interpolation operator $\cI_H^{\mathcal W}$ we infer that $\deform_j$ satisfies the isometry constraint at the nodes and the boundary conditions, which implies $\deform_j \in \cA_{H_j}$.
      The local interpolation estimate~\eqref{eq:DKTinterpolationEstimate} then implies
      \begin{align*}
        \norm{\deform_j - \deform}_{H^1(\dom;\R^3)} \leq C H^2_j \norm{\deform}_{H^3(\dom;\R^3)},
      \end{align*}
      and thus $\deform_j\to\deform$ strongly in $H^1(\dom)$. Furthermore, \eqref{eq:DKTinterpolationEstimate} also implies
      \begin{equation*}
        \|\nabla\nabla_{H_j}\deform_j-\nabla ^2\deform\|_{L^2(S)}\leq CH_j\|\nabla^3\deform\|_{L^2(S)},
      \end{equation*}
      and thus $\nabla\nabla_{H_j}\deform_j\to \nabla^2\deform$ strongly in $L^2(S)$.
      
      It remains to show that $\mathcal E^{h_j}_{H_j}(\deform_j)\to \mathcal E^0(\deform)$.
      In view of Lemma~\ref{lem:overall_convergence_quad_error} it suffices to show that
      \begin{align*}
        &\lim_{j\to \infty}\int_S \Big\langle\boldsymbol{L}^{h_j}(\mac)\nabla\nabla_{H_j}\deform_j, \nabla\nabla_{H_j}\deform_j+2n_{H_j}\otimes B^{h_j}(\mac)\Big\rangle\dd\mac\\
        &\qquad\qquad=
        \int_S \Big\langle\boldsymbol{L}^0(\mac) \nabla^2\deform, \nabla^2\deform+2n_{\deform}\otimes B^0(\mac)\Big\rangle\dd\mac
        =
        \mathcal E^0(\deform),
      \end{align*}
      but this follows again by Lemma~\ref{L:lowerbound_abstract}.
      
      \item (Recovery sequence for general $\deform$)
      
      The construction of the recovery sequence for general $\deform\in\cA$ relies on a standard diagonal-sequence argument: For $k\in\N$ choose $\deform_{k}\in\mathcal A\cap H^3(S)$ with $\|\deform_k-\deform\|_{H^2(S)}<\frac1k$. Now, for each $k$ let $\deform^{k}_{j}\in\cA_{H_j}$ be the recovery sequence of \ref{proofstep:recov_reg}. Define
      \begin{equation*}
        c^{k}_{j}
        \colonequals
        \|\deform^{k}_{j}-\deform\|_{H^1(S)}+\|\nabla\nabla_{H_j}\deform_j-\nabla^2\deform\|_{L^2(S)}+|\mathcal E^{h_j}_{H_j}(\deform_j^k)-\mathcal E^0(\deform)|.
      \end{equation*}
      By \ref{proofstep:recov_reg} we have for all $k\in\N$,
      \begin{equation*}
        \lim_{j\to\infty}c^{k}_{j}
        =
        \|\deform_{k}-\deform\|_{H^1(S)}+\|\nabla^2\deform_{k}-\nabla^2\deform\|_{L^2(S)}+|\mathcal E^0(\deform^k)-\mathcal E^0(\deform)|.
      \end{equation*}
      Lemma~\ref{L:lowerbound_abstract} implies that $\deform\mapsto \mathcal E^0(\deform)$ is continuous on $\cA$ with respect to strong convergence in $H^2(S)$. Thus,
      \begin{equation*}
        \lim_{k\to\infty}\lim_{j\to\infty}c^{k}_{j}=0.
      \end{equation*}
      By appealing to Attouch's diagonalization lemma \cite[Lemma~1.15 \& Corollary~1.16]{attouch1984variational}, there exists a diagonal sequence $(k_j)_{j\in\N}$ increasing to $+\infty$ such that $\lim_{j\to\infty}c^{k_j}_{j}=0$. We conclude that $\deform_j\colonequals\deform^{k_j}_{j}$ is a suitable recovery sequence.
      \qedhere
    \end{enumerate}
  \end{enumerate}
\end{proof}

The above result deals with the case when $(h_j,H_j)\to 0$ simultaneously.
Convergence also holds when we pass to the limits in $h$ and $H$ consecutively.
To make this claim more precise, we define the two semi-discrete energy functionals
\begin{alignat*}{2}
  \mathcal E_H & :\mathcal A_H\to\R,& \qquad \mathcal E_H(\deform_H)& \colonequals \int_S\Big\langle\boldsymbol{L}^0 \nabla\nabla_H\deform_H,\,\nabla\nabla_H\deform_H+2n_H\otimes B^0\Big\rangle\dd \mac
  \\
  \mathcal E^{h} & :\mathcal A\to\R, & \mathcal E^{h}(\deform)& \colonequals 
  \int_S \Big\langle\boldsymbol{L}^h \nabla^2\deform,\,\nabla^2 \deform+2n_{\deform}\otimes B^h\Big\rangle\dd \mac,
\end{alignat*}
and we extend these functionals to $L^2(S;\R^3)$ by setting them to $+\infty$ outside of $\cA_H$ and $\cA$, respectively.
Then we get the following statement about the individual limits and their commutativity:

\begin{theorem}[Individual limits]
  The following diagram is well-defined and commutes in the sense of $\Gamma$-convergence in the topology of $L^2(S;\R^3)$:
  \begin{center}
    \begin{tikzcd}[row sep=large, column sep=huge]
      \mathcal E^{h}_H
      \arrow[r, "H \to 0"]
      \arrow[d, "h \to 0" left]
      \arrow[dr, "{(h,H)\to 0}" description]      
      & \mathcal E^{h}
      \arrow[d, "h \to 0"] \\
      \mathcal E_H
      \arrow[r, "H \to 0" below]
      & \mathcal E^0
    \end{tikzcd}
  \end{center}
\end{theorem}
\begin{proof}
  The diagonal arrow is precisely Theorem~\ref{T:main}.
  The two convergence results for $H\to 0$ are special cases of Theorem~\ref{T:main}.
  The convergence $\mathcal E^h \to  \mathcal E^{0}$
  along the right vertical arrow is shown in \cite[Lemma~6.7]{boehnlein2023homogenized}.
  The argument there can be reused literally to also prove the left vertical arrow,
  i.e., $\mathcal E^{h}_H \to  \mathcal E_H$.
  Commutativity follows because the objects at the four corners are all defined by
  explicit closed-form expressions, and not just as limits of sequences.
\end{proof}

We conclude this section by presenting the proofs of Lemma~\ref{lem:overall_convergence_quad_error} and Lemma~\ref{L:lowerbound_abstract}.
\begin{proof}[Proof of Lemma~\ref{lem:overall_convergence_quad_error}]
  Recall the shorthand notations
  $E_S^H[f]\colonequals\cG^H_S[f]-\int_Sf\,\dd s$ and
  $E_T^H[f]\colonequals\cG^H_T[f]-\int_Tf\,\dd s$ for the global and local quadrature errors, respectively.
  The claim follows from the following two quadrature estimates%
  \begin{align}
    \label{eq:quadestimate555}
    \left|E^{H}_S\Big[\Big\langle\boldsymbol{L}^{h}(\mac)\nabla\nabla_{H}\deform_H,\nabla\nabla_{H}\deform_H\Big\rangle\Big]\right|
    & \leq
    CH\max_{T\in\macrogrid}\|\mathbb L^{h}\|_{W^{1,\infty}(T)}\|\nabla\nabla_{H}\deform_H\|^2_{L^2(S)},
    \\
    \label{eq:quadestimate555b}
    \left|E^{H}_S\Big[\Big\langle\boldsymbol{L}^{h}(\mac)\nabla\nabla_{H}\deform_H,n_{H}\otimes B^{h}(\mac)\Big\rangle\Big]\right|
    & \leq CH|S|^\frac12\max_{T\in\macrogrid}(\|\mathbb L^{h}\|_{W^{1,\infty}(T)}\|B^h\|_{W^{1,\infty}(T)})\|\nabla\nabla_{H}\deform_H\|_{L^2(S)},
  \end{align}
  which we prove now. 
  We start with \eqref{eq:quadestimate555}.
  On any element $T$ the discrete gradient $\nabla\nabla_{H}\deform_H$
  is a first-order polynomial, and thus the quadrature error estimate \eqref{eq:quadestmac0} from the appendix yields%
  \begin{align*}
    \left|E^{H}_T\Big[\Big\langle\boldsymbol{L}^{h}(\mac)\nabla\nabla_{H}\deform_H,\nabla\nabla_{H}\deform_H\Big\rangle\Big]\right|
    \leq C H_{T}\|\mathbb{L}^{h}\|_{W^{1,\infty}(T)} \|\nabla\nabla_{H}\deform_H\|_{L^2(T)}^2.
  \end{align*}
  Summation over $T\in\macrogrid$ yields \eqref{eq:quadestimate555}.
  To prove the second quadrature estimate \eqref{eq:quadestimate555b}, note that by definition of $\boldsymbol L^{h}$
  \begin{equation*}
    \Big\langle\boldsymbol L^{h}\nabla\nabla_{H}\deform_H,\,n_{H}\otimes B^{h}\Big\rangle
    =
    \Big\langle\mathbb L^{h}\big(n_{H}\cdot \nabla\nabla_{H}\deform_H\big),\,B^{h}\Big\rangle
    =
    \Big\langle\big(n_{H}\cdot \nabla\nabla_{H}\deform_H\big),\,(\mathbb L^{h})^\top B^{h}\Big\rangle,
  \end{equation*}
  pointwise on $S$. It is sufficient to estimate the quadrature error of the right-hand side.
  Since $n_{H}$ and $\nabla\nabla_{H}\deform_H$ are first-order polynomials, we may apply the quadrature estimate \eqref{eq:quadestmac0} with $f,g,\hat g$ denoting components of $(\mathbb L^{h})^\top B^{h}$, $n_{H}$, and $\nabla\nabla_{H}\deform_H$, respectively. We obtain
  \begin{equation*}
    \left|E^{H}_T\Big[\Big\langle\big(n_{H}\cdot \nabla\nabla_{H}\deform_H\big),\,(\mathbb L^{h})^\top B^{h}\Big\rangle\Big]\right|
    \leq C H_{T}\|\mathbb L^{h}\|_{W^{1,\infty}(T)}\|B^{h}\|_{W^{1,\infty}(T)}    
    \|n_{H}\|_{L^2(T)}\|\nabla\nabla_{H}\deform_H\|_{L^2(T)}.
  \end{equation*}
  A summation over $T\in\macrogrid$ and the Cauchy--Schwarz inequality yields
  \begin{align*}
    \left|E^{H}_T\Big[\Big\langle\big(n_{H}\cdot \nabla\nabla_{H}\deform_H\big),\,(\mathbb L^{h})^\top B^{h}\Big\rangle\Big]\right|
    &\leq C H\left(\max_{T\in\macrogrid}\|\mathbb L^{h}\|_{W^{1,\infty}(T)}\|B^{h}\|_{W^{1,\infty}(T)}\right)\\
    &\qquad\qquad    
    \|n_{H}\|_{L^2(S)}\|\nabla\nabla_{H}\deform_H\|_{L^2(S)}.
  \end{align*}
  Since $|n_H|\leq 1$ pointwise in $S$, \eqref{eq:quadestimate555b} follows.
\end{proof}

\begin{proof}[Proof of Lemma~\ref{L:lowerbound_abstract}]
  We only prove the statement for the case of weak convergence $G_{j}\wto G$,
  since the strong convergence is trivial. We first note that
  \begin{align*}
    \int_S\langle\widetilde{\mathbb L}_{j}G_{j},{F}_{j}\rangle \dd \mac
    \to
    \int_S\langle\widetilde{\mathbb L}G,{F}\rangle \dd s,
  \end{align*}
  since $\big(\widetilde{\mathbb L}_{j}\big)^\top F_{j}\to \big(\widetilde{\mathbb L}\big)^\top F$ strongly in $L^2(S)$.
  Hence, it suffices to consider the quadratic term. By quadratic expansion and symmetry of $\langle\mathbb L\cdot,\cdot\rangle$ we have
  \begin{align*}
    \int_S\langle\widetilde{\mathbb L}_{j}G_{j},G_{j}\rangle \dd s
    & = 2\int_S\langle\widetilde{\mathbb L}_{j}G,G_{j}\rangle \dd s-
    \int_S\langle\widetilde{\mathbb L}_{j}G,G\rangle \dd s
    +\int_S\langle\widetilde{\mathbb L}_{j}(G_{j}-G),(G_{j}-G)\rangle \dd s\\
    & \geq 2\int_S\langle\widetilde{\mathbb L}_{j}G,G_{j}\rangle \dd s-
    \int_S\langle\widetilde{\mathbb L}_{j}G,G\rangle \dd s.
  \end{align*}
  On the right-hand side we can pass to the limit, since this is product of a weakly and a strongly convergent factor. We get
  \begin{align*}
    \liminf_{h\to 0}\int_S\langle\widetilde{\mathbb L}_{j}G_{j},G_{j}\rangle \dd \mac
    \geq
    2\int_S\langle\widetilde{\mathbb L}G,G\rangle \dd s- \int_S\langle\widetilde{\mathbb L}G,G\rangle \dd \mac
    =
    \int_S\langle\widetilde{\mathbb L}G,G\rangle \dd s,
  \end{align*}
  which is the assertion.
\end{proof}

\subsection{Convergence of the fully discrete energy with the microstructure discretization of Chapter~\ref{sec:MicroNumerics}}

In this section we apply the theory of the previous section to the effective model $\dtenergy$
with a microstructure discretization as in Section~\ref{sec:MicroNumerics}.
In the following we consider microscopic and macroscopic triangulations $\microgrid$ and $\macrogrid$ as introduced in Section~\ref{sec:MicroNumerics} and Section~\ref{sec:MacroDiscretization}, respectively. In particular, we assume that the associated quadrature operators satisfy the properties claimed in Section~\ref{sec:discreteCorr} and Section~\ref{sec:discreteDKT}.

With regard to the clamped boundary condition on $\Gamma_D$, we assume the conditions of Assumption~\ref{Assumption:BoundaryData}. We note that part~\ref{item:bd_approximability} of Assumption~\ref{Assumption:BoundaryData} could be replaced by the stronger, explicit conditions of Assumption~\ref{ass:domain}.
The elasticity tensor $\mathbb L$ and the prestrain $B$ of the three-dimensional model
are supposed to satisfy the following regularity conditions:
\begin{assumption} \mbox{} 
  \begin{enumerate}[(a)]
    \item (Uniform ellipticity). $\mathbb L$ is uniformly elliptic on $S\times\Box$ in the sense of \eqref{eq:Lbounds}.
    \item \label{item:microscopic_regularity}
    (Microscopic regularity). For all $s\in S$ there holds
    \begin{equation*}
      \lim_{h\to 0}h\max_{K\in\microgrid}\Big(\|\nabla_y\mathbb L(\mac,\cdot)\|_{L^\infty(K)}+\|\nabla_y B(\mac,\cdot)\|_{L^{\infty}(K)})\Big)=0.
    \end{equation*}
    \item \label{item:macroscopic_regularity}
    (Macroscopic regularity). There exists a constant $C$ independent of $h,H$ such that
    \begin{equation*}
      \max_{T\in\macrogrid}\max_{K\in\microgrid}\Big\||\nabla_s\mathbb L|+|B|+|\nabla_s B|\Big\|_{L^\infty(T\times K)}\leq C.
    \end{equation*}
  \end{enumerate}
\end{assumption}

We note that \ref{item:microscopic_regularity} allows for discontinuous composites, as long as the discontinuities are aligned with the microscopic mesh $\microgrid$.
With regard to \ref{item:macroscopic_regularity} we note that the condition not only
allows to treat graded composites where the microstructure may change in
dependence of $\mac$ in regular way, but also the case when the dependence in $\mac$
is discontinuous, as long as the jump set is resolved by the macroscopic mesh $\macrogrid$.

\begin{corollary}[$\Gamma$-convergence of $\dtenergy$]
  In the situation described above, $\dtenergy$ $\Gamma$-converges in the topology of $L^2(S;\R^3)$ to $\widetilde{\mathcal E}^\gamma$ as $(h,H)\to 0$. In particular, any sequence of mimimizers $\deform_H^h\in\operatorname{arg\,min}\dtenergy$ converges strongly in $H^1(S;\R^3)$, up to a subsequence, to a minimizer of $\widetilde{\mathcal E}^\gamma$.
\end{corollary}

\section*{Acknowledgments}
The authors gratefully acknowledge the support by the Deutsche Forschungs\-gemeinschaft
(Funder DOI: \url{https://dx.doi.org/10.13039/501100001659})
in the Research Unit~3013 \emph{Vector- and Tensor-Valued Surface PDEs} within the sub-project \emph{TP3: Heterogeneous thin structures with prestrain}. The authors thank Martin Rumpf for helpful discussions regarding the quadrature estimate of Lemma~\ref{L:quadp1}.

\appendix

\section{Quadrature error estimates}
\label{sec:appendixA}

In this section we recall the quadrature estimates that we use in this paper.
We note that these estimates require the mesh, the finite element space and
the quadrature rule to satisfy certain conditions, which are all satisfied
in the situations of Section~\ref{sec:MicroNumerics} and Section~\ref{sec:MacroDiscretization}.
For the readers' convenience we briefly comment on these conditions and recall
the quadrature estimates in the specific form that we use in our proofs.

We first discuss the discretization of the microscopic problem introduced in Section~\ref{sec:MicroNumerics}. It consists of a hexahedral mesh $\microgrid$ of the unit cell $\Box$, the space of hexahedral finite elements $\cQ_1(\microgrid)$ and the quadrature operators $\cG^h_K$ for $K\in\microgrid$.
The quadrature operator for an element $K\in\microgrid$ is obtained by an affine transformation from a reference quadrature operator $\cG_{\hat K}[f]\colonequals\sum_{j=1}^l \hat{\omega}_j f(\hat{r}_j)$ for the reference element $\hat{K}\colonequals [0,1]^3$.
Here, $\{\hat{r}_j\}_{j=1,\ldots,l}\subset\hat{K}$ are the quadrature points and $\{\hat{\omega}_j\}_{j=1,\ldots,l}\subset(0,\infty)$ are the associated weights, and it is assumed that:
\begin{enumerate}[label=(\alph*)]
  \item \label{item:quadrature_unisolvency}
  The set $\{\hat{r}_{j}\}_{j=1,\ldots,l}$  contains a $\cQ_1(\hat{K})$-unisolvent subset.
  
  \item \label{item:quadrature_exactness}
  The quadrature rule $(\hat{r}_j,\hat{\omega}_j)$ is exact on $\cQ_1(\hat{K})$.
\end{enumerate}
Since any $K\in\microgrid$ is obtained from $\hat K$ by an affine transformation,
\ref{item:quadrature_exactness} implies that the quadrature rule $\cG_K$ is exact on $\cQ_1(K)$ for all $K\in\microgrid$.

Condition~\ref{item:quadrature_unisolvency} means that there exists a subset $J\subseteq\{1,\ldots,l\}$ such that any function in $\cQ_1(\hat{K})$ is uniquely determined by its values on the quadrature points $\{r_j\}_{j\in J}$. This condition in particular implies that $\cQ_1(\hat K)\ni g\mapsto \cG_{\hat K}[|g|^2]^\frac12$ defines a norm.
Since $\cQ_1(\hat K)$ is finite-dimensional, there exists a constant $\widehat C>0$ such that $\widehat C^{-1}\|g\|_{L^2(\hat K)}\leq \cG_{\hat K}[|g|^2]^\frac12\leq \widehat C\|g\|_{L^2(\hat K)}$ for all $g\in\cQ_1(\hat K)$.
As we assume that the family $\microgrid$ is shape-regular, we infer that 
\begin{equation}\label{eq:equivnorm}
  \forall g\in\cQ_1(K)\,:\qquad\widehat C^{-1}\|g\|_{L^2(K)}\leq \cG_{K}[|g|^2]^\frac12\leq \widehat C\|g\|_{L^2(K)},
\end{equation}
for a $\widehat C$ that can be chosen independently of $K\in\microgrid$ and the mesh size $h$.

Let $E^h_K[f]\colonequals \cG^h_K[f]-\int_K f(y)\dd y$ denote the quadrature error on
a microgrid element $K$.
The quadrature estimates used in this paper are based on two theorems from \textcite{ciarlet2002finite},
namely Theorems~4.1.4 and~4.1.5 there.
We note that these theorems are formulated for triangular mesh elements,
but they also hold for hexahedral elements, which can be checked following \cite[Exercise~4.1.7]{ciarlet2002finite}.

Under the conditions listed above, the theorems of \citeauthor{ciarlet2002finite}
have the following corollaries:

\begin{lemma}[Quadrature errors involving $\cQ_1$ functions]
  There exists a constant $\widehat C$ that only depends on the reference quadrature rule such that for all $K\in\microgrid$, $f\in W^{1,\infty}(K)$
  and $g^h, \hat g^h\in\cQ_1(K)$ we have
  \begin{align}
    \label{eq:quadest0}
    \big|E^h_K[f\partial_ig^h\partial_j\hat g^h]\big| & \leq \widehat C h\|f\|_{W^{1,\infty}(K)}\|\partial_i g^h\|_{L^2(K)}\|\partial_j\hat g^h\|_{L^2(K)},
    \qquad i,j=1,2,\\
    \label{eq:quadest2}
    \big|E^h_K[fg^h]\big| & \leq \widehat C h|K|^{\frac12}\|f\|_{W^{1,\infty}(K)}\|g^h\|_{H^1(K)}.
  \end{align}
\end{lemma}
\begin{proof}
  The bound \eqref{eq:quadest0} follows from Theorem~4.1.4 of~\cite{ciarlet2002finite} applied with $k=1$.
  Estimate \eqref{eq:quadest2} follows from Theorem~4.1.5 of~\cite{ciarlet2002finite}, with $k=1$, $n=3$, $q=\infty$, and $p=g^h$.
\end{proof}

Next we discuss the discretization of the macroscopic problem introduced in Section~\ref{sec:MacroDiscretization}.
It consists of the shape-regular family of triangulations $\macrogrid$ of $S$
and quadrature operators $\cG^H_T$ for $T\in\macrogrid$ that are exact on quadratic
polynomials $g\in\mathcal P_2(T)$, cf.~Assumption~\ref{assumption:macroQuadrature}.
Here $E^H_T[f]\colonequals \cG^T_H[f]-\int_T f(s)\dd s$ denotes the quadrature error on $T$.

Under these conditions we have:
\begin{lemma}[Quadrature errors involving $\mathcal P_1$ functions]\label{L:quadp1}
  There exists a constant $\widehat C$ that is independent of $H$,
  such that for all $T \in \macrogrid$, $f\in W^{1,\infty}(T)$
  and $g^h, \hat g^h\in\mathcal P_1(T)$ we have
  \begin{align}
    \label{eq:quadestmac0}
    \big|E^H_T[fg^h\hat g^h]\big| & \leq \widehat C H_T\|f\|_{W^{1,\infty}(T)}\|g^h\|_{L^2(T)}\|\hat g^h\|_{L^2(T)}.
  \end{align}
\end{lemma}       
\begin{proof}
  Set $\delta f\colonequals f-\frac1{|T|}\int_T f(s)\dd s$. Since the quadrature operator is exact on polynomials in $\mathcal P_2(T)$, we conclude that
  \begin{equation*}
    E^H_T[fg^h\hat g^h]=E^H_T[(\delta f)g^h\hat g^h].
  \end{equation*}
  Furthermore, by the triangle inequality and Cauchy--Schwarz inequality,
  \begin{align*}
    \big|E^H_T[(\delta f)g^h\hat g^h]\big| &\leq
    \|\delta f\|_{L^\infty(T)}\left(\cG^T_H[(g^h)^2]^\frac12\cG^T_H[(\hat g^h)^2]^\frac12+\|g^h\|_{L^2(T)}\|\hat g^h\|_{L^2(T)}\right)\\
    &\leq
    2\|\delta f\|_{L^\infty(T)}\|g^h\|_{L^2(T)}\|\hat g^h\|_{L^2(T)}.
  \end{align*}
  The claimed estimate now follows from the bound
  \begin{equation*}
    \forall s'\in T\,:\qquad |\delta f(s')|\leq \frac{1}{|T|}\int_T|f(s')-f(s)|\dd s\leq H_T\|\nabla f\|_{L^\infty(T)}.
    \qedhere
  \end{equation*}
\end{proof}

\printbibliography

\end{document}